\documentclass[pdflatex,sn-mathpys]{sn-jnl}
\usepackage[square,numbers]{natbib}
\newcommand{\lmin}{\lambda_{\min}}

\newcommand{\opt}{\mathrm{opt}}

\newcommand{\mathC}{\mathbb{C}}
\newcommand{\mathR}{\mathbb{R}}
\newcommand{\trace}{\mathrm{tr}}
\newcommand{\Ex}{\mathbb{E}}

\newcommand{\af}[1]{#1}
\newcommand{\rGS}{\mathrm{rGS}}
\newcommand{\GSW}{\mathrm{GSW}}
\renewcommand{\r}{\mathrm{ra}}
\newcommand{\K}{\mathrm{K}}
\newcommand{\g}{\mathrm{gr}}
\normalbaroutside

\theoremstyle{thmstyleone}%
\newtheorem{theorem}{Theorem}%  meant for continuous numbers
\newtheorem{lemma}[theorem]{Lemma}% 
\newtheorem{corollary}[theorem]{Corollary}%
\newtheorem{definition}[theorem]{Definition}%

\begin{document}

\title[Randomized and Greedy Relaxations for Linear Systems]{On the Convergence of Randomized 
and Greedy Relaxation Schemes for Solving Nonsingular Linear Systems of Equations}

\author[1]{\fnm{Andreas} \sur{Frommer}}\email{frommer@uni-wuppertal.de}
\author[2]{\fnm{Daniel B.} \sur{Szyld}}\email{szyld@temple.edu}
% Authors: full names plus addresses.
\affil[1]{\orgdiv{Department of Mathematics}, \orgname{Bergische Universit\"at Wuppertal}, \postcode{42097}, \orgaddress{\city{Wuppertal}, \country{Germany}}}
\affil[2]{\orgdiv{Department of Mathematics}, \orgname{Temple University}, \orgaddress{\city{Philadelphia}, \state{PA}, \postcode{19122-6094}, \country{USA}}}

\abstract{We extend results known for the randomized Gauss-Seidel and the Gauss-Southwell methods for the case of a Hermitian and positive definite matrix to certain classes of non-Hermitian matrices. We obtain convergence results for a whole range of parameters describing the probabilities in the randomized method or the greedy choice strategy in the 
Gauss-Southwell-type methods. We identify those choices which make our convergence bounds best possible. Our main tool is to use weighted $\ell_1$-norms to 
measure 
%the size of 
the residuals. A major result is that the best convergence bounds that we obtain for the expected values in the randomized algorithm are as good as the best for the deterministic, but more costly algorithms of Gauss-Southwell type.   
\af{
Numerical experiments illustrate the convergence of the method and the bounds obtained. Comparisons with the
randomized Kaczmarz method are also presented.
}
}



\maketitle

\section{Introduction}
Classical stationary iterations such as Jacobi and Gauss-Seidel 
(see, e.g., \cite{BermanPlemmons,saad-book,varga-62}) \af{to solve a square linear system
\begin{equation} \label{linear_system:eq}
A x = b, \mbox{ where } A \in \mathC^{n\times n}, \enspace x, b \in \mathC^{n},
\end{equation}
}%
nowadays
are found to be useful in many situations,
such as smoothers for multigrid methods (see, e.g.,\cite{Ruede.book,Trottenberg}), in high performance computing 
(see, e.g.,~\cite{Wolfson-Pou.Chow.17}) and in particular as scafolding for
methods for discretized PDEs based on domain decomposition;
see, e.g.,~\cite{Glusa.etal.20,Magoules.Szyld.Venet.17,Smith:1996:DPM}.

\af{In recent years, randomized algorithms have gained a lot of attention in numerical computation; see, e.g., the surveys
\cite{Kannan.Vempala.17} and \cite{Martinsson.Tropp.20}. Many different randomized methods and 
algorithms have in particular been suggested and analyzed for the solution of consistent and non-consistent, 
square and non-square linear systems; see, e.g., \cite{Gower.Richtark.15}. 
%and references therein,
%in particular 
%%\textbf{Check where this citations fits now: \cite{Ma.Needell.Ramdas.SIAMX15}.}
These methods are 
attractive in situations which typically arise in an HPC or a data science context when matrix products are considerably expensive or when the matrix is so 
large that it does not fit in main memory; see, e.g., the 
discussion in the recent paper \cite{Goweretal2021}.
}

\af{For linear systems, the emphasis has been so far on randomized coordinate descent type algorithms. These methods aim at finding the minimizer of a convex functional $f: \mathC \to \mathC$, the minimizer being the solution of the linear system. The methods perform a sequence of relaxations, where in each relaxation a coordinate $i$ is chosen at random and the current iterate $x$ is modified to become $x + te_i, e_i$ the $i$-th unit vector and $t$ such that $f(x+te_i)$ is minimal. Many convergence results on randomized coordinate descent methods are known, see \cite{Gower.Richtark.15, RichtarikKavac2020}, e.g., and the typical assumption is that $f$ is at least differentiable.  
}

\af{
If $A \in \mathC^{n \times n}$ in \eqref{linear_system:eq} is 
Hermitian and positive definite (hpd), and we take $f(x) = x^* Ax - 
2x^*b $, minimizing in coordinate $i$ is equivalent to solving 
equation $i$ of~\eqref{linear_system:eq} with respect to $x_i$. 
The resulting randomized coordinate descent method is thus a 
randomized version of the Gauss-Seidel method, and an analysis 
of this randomized method was given in 
\cite{Leventhal.Lewis.10} and also, in the more general context 
of randomized Schwarz methods in Hilbert spaces, in 
\cite{GriOsw2012}. See also \cite{Avron.JACM.15}.
}

\af{
If $A$ in \eqref{linear_system:eq} is not hpd, the typical approach is to obtain a randomized algorithm is to consider randomized 
Kaczmarz methods, i.e., coordinate descent for one of the normal equations $A^*Ax = A^*b$ or $AA^*y = b$ with corresponding 
convex and differentiable functionals $f_1(x) = x^*A^* A x - 
2x^*A^*b$ and $f_2(y) = y^*AA^*y - 2y^*b$, respectively. This 
approach can also be pursued when $A$ is non-square, and the 
system may be consistent or non-consistent. Note that the 
original Kaczmarz method \cite{Kaczmarz1937} corresponds to 
coordinate descent for $f_2$ with an integrated back-transformation 
from the iterate $y$ to $x=A^*y$. We emphasize 
that while coordinate descent for $f_1$ is often also termed 
``Gauss-Seidel'', it is different from classical Gauss-Seidel 
directly applied to \eqref{linear_system:eq}, which is what we 
focus on in this paper. There is a 
tremendous amount of literature dealing with randomized 
Kaczmarz type algorithms which we cannot cite exhaustively 
here. Recent publications include 
\cite{
BaiWu18,
BaiWu18b,
BaiWu21,
BaiWangWu21,
BaiWangMuratova22,
K_Du_NLAA19,
Goweretal2021,
GuanGuoXingQiao20,
GuoLi18,
HaddockMa21, Ma.Needell.Ramdas.SIMAX15,
Steinberger21,Steinberger21b,WangLiBaoLiu22,Yang21}.
}

\af{Now, when $A$ is square and non-singular, considering a 
randomized version of the Gauss-Seidel method applied to the 
linear system \eqref{linear_system:eq} directly, is an 
attractive alternative to the randomized Kaczmarz type 
approaches, since under appropriate conditions on $A$ it 
converges more rapidly and requires less work per iteration. This 
is known to be the case when $A$ is hpd; see the papers \cite{Leventhal.Lewis.10} and 
\cite{GriOsw2012} mentioned earlier. The main new contribution 
of  the present work is a convergence analysis for randomized 
Gauss-Seidel also for the case when $A$ is generalized 
diagonally dominant. An important methodological aspect of our 
work is that for $A$  generalized diagonally dominant we do not 
directly relate Gauss-Seidel for \eqref{linear_system:eq} to an 
equivalent coordinate descent method for an appropriate convex 
functional. Our technique of proof will, nevertheless, rely on 
showing that the iterates $x$ reduce---but, as opposed to 
gradient descent, do not necessarily minimize---a weighted 
$\ell_1$-norm of the residual $b-Ax$. Note that 
$\ell_1$-norms are not differentiable and that weighted $\ell_1$ (and 
$\ell_\infty$) -norms) arise canonically in the context of 
generalized diagonal dominance; see, e.g., \cite{
BermanPlemmons, 
varga-62}. 
As a `by-product'' of our analysis, we will also obtain convergence results of greedy choice algorithms of Gauss-Southwell type, see below.
}

%One notable exception is the work by Griebel and Oswald \cite{GriOsw2012}, where they consider
%Hermitian positive definite systems and Schwarz methods in (infinite-dimensional) function spaces, and from which we
%draw some inspiration here; see also~\cite{Avron.JACM.15}.
%We also mention randomized incomplete Cholesky factorizations for Hermitian positive definite matrices~\cite{Kyng.etal.16,Spielman.Cheng.14}, 
%in the context of theoretical computer science. These factorizations might be used as
%preconditioners for Conjugate Gradient, but they seem not competitive with other preconditioners~\cite{Chen.Liang.Biros.21}.

%In this paper we study randomized versions of certain
%classical methods for the solution of nonsingular systems of the form
%\begin{equation} \label{linear_system:eq}
%A x = b, \mbox{ where } A \in \mathC^{n\times n}, \enspace x, b \in \mathC^{n},
%\end{equation}
%and where we
%denote by $x^*$ the solution of~\eqref{linear_system:eq}.
%where, as it is common use, we actually use $x^*$ to denote its solution $x^*=A^{-1}b$.
\af{In this paper, we consider general methods} based on a splitting $A = M-N$ so that with $H = M^{-1}N$ and $c=M^{-1}b$ one
obtains the (affine) fixed point iteration
\begin{equation} \label{generic_linear:eq}
\mbox{choose } x^0, \enspace x^{m+1} = Hx^{m} + c, \enspace {m}=0,1,\ldots
\end{equation}
as an iterative solution method for $Ax=b$.  In particular, the 
solution $x^*$ of~\eqref{linear_system:eq} is a fixed point of~\eqref{generic_linear:eq}.
\af{
Usually, the matrix $H$ is never formed. Instead, a linear system with the coefficient matrix $M$ is solved at
each iteration $m$.} 
\af{In this general splitting framework, the Gauss-Seidel method is characterized by $M$ being} the lower triangular part of $A$, and the method is equivalent to
relaxing one row at a time in the natural order $1,2,\ldots,n$. In other terms, 
if we write 
the classical Jacobi splitting
$A = D -B$, $D$ being the diagonal of $A$, then, with $H = D^{-1}B = I - D^{-1}A$,
we have the following rendition of the Gauss-Seidel algorithm, using a ``global'' index $k$ for each single relaxation.

\begin{algorithm}[H]
\begin{algorithmic}
\For{$k=1,2,\ldots$ until a convergence criterion is satisfied}
 \State $i = k - n \lfloor (k-1)/n \rfloor $   \Comment{$i \in \{1,\ldots,n$\} with $i \equiv k \bmod n$} 
 \State $x_i^{k+1} = \sum_{j=1}^n h_{ij}x_j^k + c_i, \quad x_\ell^{k+1} = x_\ell^k$  for $\ell \neq i$
\EndFor
\end{algorithmic}
\caption{Sequential relaxation for \eqref{generic_linear:eq} (``Gauss-Seidel'' if $H=D^{-1}B$)  \label{sequential:alg}}
\end{algorithm}

Each update from $k$ to $k+1$ is termed one {\em relaxation}, and $n$ such relaxations, {since they are done one after the other},
correspond to one iteration ${m}$ in~\eqref{generic_linear:eq} with $H = (D-L)^{-1}U$, where $-L$ and $-U$ denote the lower and upper triangular part of $A$, respectively ({and} $B = L+U)$. So,
$n$ successive relaxations of a ``Jacobi-type", i.e.\ with $H=I-D^{-1}B$, performed in the natural order $1,2,\ldots,n$, is identical to
one iteration of Gauss-Seidel. 

Gauss in fact proposed another method, later popularized by Southwell, 
and known either as Southwell method or as Gauss-Southwell method; 
see, e.g., the historic review %beautifully written history 
of these developments in 
\cite{Saad.history.20}.
We describe this method {in Algorithm~\ref{greedy_linear:alg}} for a general splitting based method where $H = M^{-1}N$ with $A = M-N$. The Gauss-Southwell method arises for $H = D^{-1}B$. % using a fixed parameter $\beta \in (0,1]$. 
The original Gauss-Southwell method selects the component $i$ to be relaxed as the one at which the residual $r^k = b-Ax^k$ is largest, i.e., it takes $i$ for which  
\begin{equation} \label{standard_pick:eq}
 |r_i^k| = \max_{j=1}^n |r_j^k|. 
\end{equation}
As a consequence, we are not updating components in a prescribed order, but rather 
choose the row to relax to be the one for which the current residual has its largest component. 
This can be considered a greedy pick strategy, and we formulate Algorithm~\ref{greedy_linear:alg} in a manner to allow for general greedy pick rules.

\begin{algorithm}[H]
\begin{algorithmic}
\State fix a greedy pick rule, for example \eqref{standard_pick:eq}
\For{$k=1,2,\ldots$}
 \State determine an index $i$ according to the greedy pick rule
 \State $x_i^{k+1} = \sum_{j=1}^n h_{ij}x_j^k + c_i, \quad x_\ell^{k+1} = x_\ell^k$  for $\ell \neq i$
\EndFor
\end{algorithmic}
\caption{Greedy relaxation for \eqref{generic_linear:eq} (``Gauss-Southwell'' if $H = D^{-1}B$) \label{greedy_linear:alg}}
\end{algorithm}

A generalization of the greedy pick rule \eqref{standard_pick:eq} is to fix weights $\beta_i > 0$ and choose~$i$ such that
\begin{equation} \label{weighted_pick:eq}
\beta_i |r_i^k| = \max_{j=1}^n \beta_j |r_j^k|,
\end{equation}
{\color{blue}
cf.\ \cite{GriOsw2012},
}
and we will see later that for appropriate choices of the $\beta_i$ {we can prove better convergence bounds than for the standard greedy pick rule \eqref{standard_pick:eq}.}

Both greedy pick rules \eqref{standard_pick:eq}
and \eqref{weighted_pick:eq}
require 
an update of the residual after each relaxation, which represents extra work. % as compared to randomized relaxation. 
Moreover, additional work is required for computing the maximum. 
Using the ``preconditioned'' residual 
$\hat{r}^k = M^{-1}r^k = c - (I-H)x^k$ can to some extent reduce this overhead: Once $\hat r^k$ is computed, the next relaxation $x_i^{k+1} = \sum_{j=1}^n h_{ij}x_j^k + c_i$ can be obtained easily as
\[
x_i^{k+1} = x_i^k + \hat{r}_i^k.
\]
We might therefore want to use a greedy pick rule based on the preconditioned residuals, which using fixed weights $\beta_i$ as before, can be formulated as
\begin{equation} \label{prec_pick:eq}
\beta_i |\hat r^k_i| = \max_{j=1}^n \beta_j |\hat r_j^k|.
\end{equation}

It has been demonstrated that Gauss-Southwell can indeed converge in fewer relaxations than Gauss-Seidel, but
the total computational time is often higher, due to the computation of the maximum,
and, in a parallel setting, the added cost of communication \cite{Wolfson-Pou.Chow.17}.

In this paper, we discuss the greedy relaxation scheme of Algorithm~\ref{greedy_linear:alg} as well as a randomized version of Algorithm~\ref{sequential:alg}, which
for $H=D^{-1}B$ is usually called randomized Gauss-Seidel. 
We give bounds on the expected value of the residual norm which match analogous convergence bounds for the greedy algorithm.

The randomized iteration derived from \eqref{generic_linear:eq} fixes probabilities $p_i \in (0,1)$, 
\linebreak
$i=1,\ldots,n$, with $\sum_{i=1}^n p_i = 1$ and proceeds as follows.  
\begin{algorithm}[H]
\begin{algorithmic}
\For{$k=1,2,\ldots$}
 \State choose index $i$ with probability $p_i$ 
 \State $x_i^{k+1} = \sum_{j=1}^n h_{ij}x_j^k + c_i, \quad x_\ell^{k+1} = x_\ell^k$  for $\ell \neq i$
\EndFor
\end{algorithmic}
\caption{Randomized iteration for \eqref{generic_linear:eq} (``randomized Gauss-Seidel'' if $H=D^{-1}B$) \label{randomized_linear:alg}}
\end{algorithm}

%\newpage
In the randomized Algorithm~\ref{randomized_linear:alg}, the order of the relaxation does not follow a prescribed order, as in Gauss-Seidel,
nor a greedy order depending on the entries in the current residual vector, as in Gauss-Southwell,
but instead, each row $i$ to relax is chosen at random with a fixed positive probability $p_i$.

The paper is organized as follows.
\af{We first repeat the convergence results from \cite{GriOsw2012,Leventhal.Lewis.10}
for matrices which are hpd
in Section~\ref{hpd:sec},} and then, in the rest of the paper,  present new results for 
non-Hermitian matrices.
As a byproduct of our investigation, we also show that some greedy choices other
than~\eqref{standard_pick:eq}
produce methods {for which we obtain better bounds on the rate of convergence.}

%For $A$ hpd, the natural norm for the analysis  is the $A$-norm. When working with rectangular matrices,
%one finds, e.g., in \cite{K_Du_NLAA19, Ma.Needell.Ramdas.SIMAX15} that the Frobenious norm is
%used, or sometimes the $A^TA$ norm.
%One of our contributions is to use a weighted $\ell_1$-norm, for which convergence bounds
%can be obtained.
\af{Our results are theoretical in nature and illustrated with
numerical experiments. We are aware that other methods may be more efficient than those discussed here. But we believe that our results represent an interesting contribution to randomized and greedy relaxation algorithms since they show that we can deviate from the slowly converging Kaczmarz type approaches not only when $A$ is hpd but also when $A$ is (generalized) diagonally dominant. We thus trust that on one hand, the results are interesting in and on themselves, and on the other they may form the basis
for the analysis of other practical methods.  
}
Asynchronous iterative methods \cite{Frommer.Szyld.00}, \af{for example}, can be {interpreted in terms of} 
randomized iterations; see, e.g., 
\cite{Avron.JACM.15, strik:02}. 
\af{We expect} that the theoretical tools developed {here can also serve as a foundation}
for the analysis of such asynchronous methods,
as well as for randomized block methods and randomized Schwarz methods for nonsingular
linear systems; cf.~\cite{GriOsw2012}.

For future use we recall that $r^k = A(x^*-x^k)$ where $x^*$ is the solution of the linear system \eqref{linear_system:eq}
and $x^*-x^k$ is the error at the relaxation (or iteration) $k$.

\section{The Hermitian and positive definite case} \label{hpd:sec}
Consider the particular case that $H$ in 
Algorithms~\ref{greedy_linear:alg} and~\ref{randomized_linear:alg} 
arises from the (relaxed) Jacobi splitting
\begin{equation} \label{JOR_splitting:eq}
A = \tfrac{1}{\omega}D-\left( (1-\tfrac{1}{\omega})D +B \right) \mbox{ with $D$ the diagonal part of $A = D-B$,}
\end{equation}
where $\omega \in \mathR$ is a relaxation parameter, i.e.
\begin{equation} \label{JOR_matrix:eq}
H = H_\omega = (1-\omega)I + \omega D^{-1} B.
\end{equation}
Then the fixed point iteration \eqref{generic_linear:eq} is just the relaxed Jacobi iteration, which reduces to standard 
Jacobi if $\omega = 1$, {and the associated randomized iteration from Algorithm~\ref{randomized_linear:alg} is the randomized relaxed Gauss-Seidel method whereas the associated greedy Algorithm~\ref{greedy_linear:alg} is known as the relaxed Gauss-Southwell method if we take the greedy pick 
rule \eqref{standard_pick:eq}.}  
Using the residual $r^k$, the update in the third lines of 
Algorithms~\ref{greedy_linear:alg} and~\ref{randomized_linear:alg} for $H_\omega$
can alternatively be formulated as
\begin{equation} \label{JOR_residual_formulation:eq}
x^{k+1} = x^k + \omega \cdot \frac{r^k_i}{a_{ii}} e_i,
\end{equation}
where $e_i$ denotes the $i$th canonical unit vector in $\mathC^n$.

Now assume that $A$ is hpd. The Jacobi iteration then does not converge unconditionally, a sufficient condition for convergenc being that with $A = D-B$ the matrix $D+B$ is hpd 
as well. The relaxed Gauss-Seidel iteration, on the other hand, is unconditionally convergent provided $\omega \in (0,2)$; see, e.g., \cite{varga-62}. 

For the randomized Gauss-Seidel and the Gauss-Southwell iterations the following results are essentially known. 

In fact, most of Theorem~\ref{GS_hpd:thm} is a special case of what was shown in \af{\cite{Leventhal.Lewis.10}} and \cite{GriOsw2012} in the fairly more general context 
of (relaxed) randomized multiplicative Schwarz methods. 
%As is discussed in \cite{GriOsw2012}, in the finite-dimensional 
%setting of the Gauss-Seidel iteration the general proof simplifies considerably. 
For the sake of completeness, \af{and to set the stage for our new results}, we repeat the essentials of the proofs in \cite{Leventhal.Lewis.10} and \cite{GriOsw2012} here. 

We use the $A$-inner product and the $A$-energy norm which, for $A$ hpd, are defined as 
\[
\langle x,y \rangle_A = \langle Ax,y \rangle, \enspace \|x\|_A = \sqrt{\langle x,x\rangle}_A,
\]
with $\langle \cdot, \cdot \rangle$ the standard inner product on $\mathC^n$.
Before we state the main theorem, we formulate the following useful 
result relating the harmonic and the arithmetic means of a sequence and the extrema of the product sequence. 
%fact two facts about convex combinations and means which we will useful in our proofs. 

\begin{lemma} \label{convex_comb:lem}Let $a_i, \gamma_i \in \mathR, i=1,\ldots, n$  with $a_i > 0, \gamma_i \geq 0 $, $i=1,\ldots,n$. Then
%\begin{itemize}
%\item[(i)] If $\sum_{i=1}^n \gamma_i = 1$, then $ \min_{i=1}^n a_i \leq \sum_{i=1}^n \gamma_i a_i \leq \max_{i=1}^n  a_i$.
%\item[(ii)] 
%If $a_i > 0, i=1,\ldots,n$, we also have  
\[
\min_{i=1}^n \gamma_i a_i \leq \alpha\gamma    \leq \max_{i=1}^n \gamma_i a_i, 
\]
where
\[
\alpha  = n/\sum_{j=1}^n a_j^{-1}, \enspace  \gamma = \frac{1}{n} \sum_{i=1}^n \gamma_i
\]
are the harmonic mean and the arithmetic mean, respectively. 
%\end{itemize}
\end{lemma}
\begin{proof} Take the special convex combination of the $a_i$ with coefficients $\hat{\gamma}_i =  \frac{\alpha/n}{a_i}$. Then
$a_i \hat{\gamma_i} =  \alpha/n$, and for the convex combination of the $a_i$ with coefficients 
$\tilde{\gamma}_i = \frac{\gamma_i}{n\gamma}$  there is at least one index, say $j_0$, for which
$\tilde{\gamma}_{j_0} \geq  \hat{\gamma}_{j_0}$, since otherwise 
$\tilde{\gamma}_i < \hat{\gamma}_i$ for all~$i$ and thus $\sum_{i=1}^n \tilde{\gamma}_i < \sum_{i=1}^n\hat{\gamma}_i = 1$. This proves  $\max_{i=1}^n \gamma_i a_i \geq \alpha/n \cdot n\gamma = \alpha\gamma$. The inequality for the minimum follows in a similar manner.
\end{proof}

\begin{theorem} \label{GS_hpd:thm}  Let $A$ be hpd and denote $\lmin > 0$ its smallest eigenvalue.
\begin{itemize}
\item[(i)]  In randomized relaxed Gauss-Seidel (Algorithm~\ref{randomized_linear:alg} with $H = H_\omega = (1-\omega)I + \omega D^{-1} B$) the expected values for the squares of the norms of the errors $e^k = x^k-x^*$ satisfy
\begin{equation} \label{expected_val_hpd:eq}
\mathbb{E}\left(\|x^k-x^*\|_A^2\right) \leq (1- \alpha^\rGS)^k \| x^0-x^*\|_A^2 
\end{equation}
with
\begin{equation*}
 \alpha^\rGS = \omega(2-\omega)\lambda_{\min} \min_{i=1}^n \frac{p_i}{a_{ii}} \cdot
\end{equation*}
Herein, $1-\alpha^\rGS$ becomes smallest 
%the factor $1-\omega(2-\omega)\lambda_{\min} \min_{i=1}^n \frac{p_i}{a_{ii}} $ becomes smallest 
if we take 
%$p_i = \tfrac{a_{ii}}{\trace(A)}$ 
$p_i = {a_{ii}}/{\trace(A)}$ 
for all $i$, in which case \eqref{expected_val_hpd:eq} holds with 
\begin{equation} \label{alpha_opt:eq}
\alpha^\rGS =  \alpha_{\opt} = \omega(2-\omega)\frac{\lambda_{\min}}{\trace(A)} ~\cdot
\end{equation}
\item[(ii)] In relaxed Gauss-Southwell (Algorithm~\ref{greedy_linear:alg} with $H = H_\omega = (1-\omega)I + \omega D^{-1} B$) and the greedy pick \eqref{weighted_pick:eq} we have 
\begin{equation} \label{bound_greedy_hpd:eq}
\|x^{k}-x^*\|_A^2 \leq  
(1- \alpha^\GSW )^k \|x^{0}-x^*\|_A^2 
\end{equation}
with
\begin{equation*}\alpha^\GSW = \omega(2-\omega)\lambda_{\min} \min_{i=1}^n \frac{\pi_i}{a_{ii}}, \enspace \pi_i = \frac{1/\beta_i^2} {\sum_{j=1}^n 1/\beta_j^2} 
%\frac{\min_{j=1}^n \beta_j^2}{\max_{i=1}^n \beta_i^2 a_{ii}} \cdot \frac{\lambda_{\min}}{n} 
~\cdot
\end{equation*}
Herein, $1-\alpha^\GSW$ becomes smallest if we take  
$\beta_i = {1}/{\sqrt{a_{ii}}}$ 
for all $i$ in the greedy pick rule, i.e., we choose $i$ such that
\begin{equation} \label{hpd_greedy_opt:eq}
\frac{|r_i^k|^2}{a_{ii}} = \max_{j=1}^n \frac{|r_j^k|^2}{a_{jj}},
\end{equation}
in which case \eqref{bound_greedy_hpd:eq} holds with 
$\alpha^{\GSW}$ the same optimal value as in (i), i.e.\ $\alpha^\GSW =  \alpha_{\opt}$ from \eqref{alpha_opt:eq}. 
\end{itemize}
\end{theorem}
\begin{proof} 
If in relaxation $k$ we choose to update component $i$, then 
\begin{eqnarray}
\| x^{k+1}-x^*\|_A^2% 
&=& \| x^k-x^*\|_A^2 + 2 \Re  \langle x^k-x^*, \tfrac{\omega r^k_i}{a_{ii}} e_i \rangle_A  +
     \langle \tfrac{\omega r^k_i}{a_{ii}} e_i, \tfrac{\omega r^k_i}{a_{ii}} e_i \rangle_A \nonumber \\
&=& \| x^k-x^*\|_A^2 - 2 \Re\langle r^k, \tfrac{\omega r^k_i}{a_{ii}} e_i \rangle  +
     \langle \tfrac{\omega r^k_i}{a_{ii}} Ae_i, \tfrac{\omega r^k_i}{a_{ii}} e_i \rangle \\
&=& \| x^k-x^*\|_A^2 - \tfrac{2\omega}{a_{ii}} |r^k_i|^2  + \tfrac{\omega^2}{a_{ii}^2} |r^k_i|^2a_{ii} \nonumber \\
&=& \| x^k-x^*\|_A^2 - \tfrac{\omega(2-\omega)}{a_{ii}} |r^k_i|^2. \label{update_hpd:eq}
\end{eqnarray}
Therefore, in randomized Gauss-Seidel the expected value for 
$\| x^{k+1}-x^*\|_A^2$, conditioned to the given value for $x^k$, is
\begin{eqnarray*}
\Ex \left(\| x^{k+1}-x^*\|_A^2 \, \big| \, x^k \right) &=& \sum_{i=1}^n p_i \left( \| x^k-x^*\|_A^2-\tfrac{\omega(2-\omega)}{a_{ii}} |r^k_i|^2 \right) \\
&=& \| x^k-x^*\|_A^2 -  \omega(2-\omega) \sum_{i=1}^n \frac{p_i}{a_{ii}} |r^k_i|^2 \\
&\leq& \| x^k-x^*\|_A^2 -  \omega(2-\omega) \min_{i=1}^n \frac{p_i}{a_{ii}} \|r^k\|^2 \\
&\leq&  \left(1-\omega(2-\omega)\lmin  \min_{i=1}^n \frac{p_i}{a_{ii}} \right) \| x^k-x^*\|_A^2 ~,
\end{eqnarray*}
with the last inequality holding due to $\langle x^k-x^*,x^k-x^* \rangle_A = \langle r^k, A^{-1}r^k \rangle  
\leq \tfrac{1}{\lmin}  \langle r^k, r^k \rangle $. This gives \eqref{expected_val_hpd:eq}. If we have 
%$p_i = \tfrac{a_{ii}}{\trace(A)}$, 
$p_i = {a_{ii}}/{\trace(A)}$, 
%then $\min_{i=1}^n \frac{p_i}{a_{ii}} = \tfrac{1}{\trace(A)}$ 
then $\min_{i=1}^n ({p_i}/{a_{ii}}) = {1}/{\trace(A)}$, and this is larger or equal than 
%$\min_{i=1}^n \frac{p_i}{a_{ii}}$ 
$\min_{i=1}^n ({p_i}/{a_{ii}})$ 
for any choice of the probabilities $p_i$ by Lemma~\ref{convex_comb:lem}. This gives the second statement in part (i).

To prove part (ii) we observe that from the greedy pick rule $\beta_i |r_i^k| \geq  \max_{j=1}^n \beta_j |r_j^k|$  we have $\beta_i^2 |r^k_i|^2/\|r^k\|^2 \geq \beta_j^2 |r^k_j|^2/\|r^k\|^2$ for all $j$ which, using Lemma~\ref{convex_comb:lem} (with $\gamma_j = |r_j^k|^2$) gives
\[
  \frac{|r_i^k|^2}{a_{ii}} = \frac{\beta_i^2 \cdot  |r_i^k|^2}{\beta_i^2 a_{ii}} 
  = \frac{1}{\beta_i^2 a_{ii}} \max_{j=1}^n  \beta_j^2 |r_j^k|^2 \geq \frac{1}{\beta_i^2 a_{ii}} \frac{\|r^k\|^2}{\sum_{\ell=1}^n 1/\beta_\ell^2}~, 
\]
from which we deduce
\[
 \frac{|r_i^k|^2}{a_{ii}} \geq \min_{j=1}^{n} \frac{\pi_j}{a_{jj}} \|r^k\|^2 \mbox{ with } \pi_j=  \frac{1/\beta_j^2}{\sum_{\ell=1}^n 1/\beta_\ell^2} ~ \cdot
\]
So \eqref{update_hpd:eq} this time yields
\[
\|x^{k+1}-x^*\|_A^2 \leq \left( 1-\omega(2-\omega)\lambda_{\min} \min_{j=1}^n \frac{\pi_j}{a_{jj}} \right) \|x^k-x^*\|_A^2 ~,
\]
which results in \eqref{bound_greedy_hpd:eq}.
Finally, using Lemma~\ref{convex_comb:lem} (with $\gamma_i = \pi_i$) we obtain
\[
 \min_{i=1}^n \frac{\pi_i}{a_{ii}} \leq \frac{1}{\trace(A)} ~, 
\]
and for the choice $\beta_i = {1}/{\sqrt{a_{ii}}}$ we have $\pi_i = a_{ii}/\trace(A)$ and thus 
\[
 \frac{\pi_i}{a_{ii}}  = \frac{1}{\trace(A)} \mbox{ for } i=1,\ldots,n. %~~~~ \cvd
\]
\end{proof}

We note that if in randomized Gauss-Seidel we 
choose all probabilities to be equal, $p_i = {1}/{n}$ for all $i$, then 
\[
\alpha^\rGS = \omega(2-\omega)\frac{\lambda_{\min}}{n \max_{i=1}^n a_{ii}}
\]
in \eqref{expected_val_hpd:eq}, which is smaller than $\alpha_{\opt}$ unless all diagonal elements $a_{ii}$ are equal. 
\af{We have a completely analogous situation for Gauss-Southwell: If we take the unweighted greedy pick rule 
\eqref{standard_pick:eq}, we obtain a value for $\alpha^{\GSW}$ 
which, interestingly, is the same than $\alpha^{\rGS}$ for 
randomized Gauss-Seidel with uniform probabilities. And this 
value is smaller than the value $\alpha_{\opt}$ that we 
obtain for the weighted greedy pick rule 
\eqref{hpd_greedy_opt:eq}, a value which is, interestingly 
again, equal what we obtain as the maximum value for $\alpha^{\rGS}$ in the randomized method (with the weighted probabilities $p_i = a_{ii}/\trace(A)$). }
%Interestingly, if in 
%Gauss-Southwell we choose to relax the component for which $|r_i^k|$ is maximal, i.e., we take the greedy pick rule
%\eqref{standard_pick:eq}---which amounts to setting $\beta_i=1$ for all $i$---we obtain
%the identical value 
%$$ \alpha^\GSW = \omega(2-\omega)\frac{\lambda_{\min}}{n \max_{i=1}^n a_{ii}}$$
%in \eqref{bound_greedy_hpd:eq}.  

Note also that if one scales the hpd matrix $A$ symmetrically so that it has unit diagonal, then the
greedy pick~\eqref{hpd_greedy_opt:eq} in Theorem~\ref{GS_hpd:thm} (ii) reduces to the standard
Gauss-Southwell pick~\eqref{standard_pick:eq}: Let $G= D^{-1/2} A D^{-1/2}$, then, 
the system~\eqref{linear_system:eq} is equivalent to $Gy =  D^{-1/2}b$ with the change of variables
$x = D^{1/2}y$. Running Algorithm~\ref{greedy_linear:alg} for $A$ and $b$ in the variables $x$ with the greedy pick \eqref{weighted_pick:eq} with $\beta_i = 1/\sqrt{a_{ii}}$ 
is equivalent to running the same algorithm for $G$ and $D^{-1/2}b$ in the variables $y$
with the standard greedy pick \eqref{standard_pick:eq}.
One can then express the bounds of the theorem in the scaled variables in the appropriate energy norm, since we have
$
\|y^k - y^* \|_G = \| D^{1/2} y^k - D^{1/2} y^* \|_A = \| x^k - x^* \|_A ;
$
cf.~\cite[Section 3.1]{Avron.JACM.15}.

\subsection*{Numerical example} \label{hpd_example:subsec}
Throughout this paper we give illustrative numerical examples based on the convection-diffusion equation for a concentration $c = c(x,y,t): [0,1] \times [0,1] \times [0,T] \to \mathC$
\begin{eqnarray} \label{advec_diff:eq}
\frac{\partial }{\partial t} c &=& - \frac{\partial }{\partial x} \alpha \frac{\partial }{\partial x} c
- \frac{\partial }{\partial y} \beta \frac{\partial }{\partial y} c + \frac{\partial }{\partial x} \nu c + \frac{\partial }{\partial y} \mu c, \\
& & c(x,y,t) \, =\, 0 \mbox{ if } (x,y) \in \partial ([0,1] \times [0,1]), \nonumber \\
& & c(x,y,0) \, =\, c_0(x,y) \mbox{ for } (x,y) \in [0,1] \times [0,1]. \nonumber
\end{eqnarray}
The positive diffusion coefficients $\alpha$ and $\beta$ are allowed to depend on $x$ and $y$, $\alpha = \alpha(x,y), \beta = \beta(x,y)$, and this also holds for the velocity field $(\nu,\mu) = (\nu(x,y),\mu(x,y))$. We discretize in space using standard finite differences with $N$ interior equispaced grid points in each direction. This leaves us with the semi-discretized system 
\[
\frac{\partial }{\partial t} c = B c, \enspace c(x,y,0) = c_0(x,y) \mbox{ for all grid points } x,y,
\]
where now $c=c(t)$ is a two-dimensional array, each component corresponding to one grid point. 
Using the implicit Euler rule as a symplectic integrator means that at a given time $t$ and a stepsize $\tau$ we have to solve
\begin{equation} \label{the_system:eq}
(\underbrace{I+\frac{\tau}{2}B}_{=: A})c(t+1) = \frac{\tau}{2} B c(t)
\end{equation}
for $c(t+1)$.
We illustrate the covergence behavior of the Gauss-Seidel variants considered in this paper when solving the system \eqref{the_system:eq} for appropriate choices of $\alpha,\beta, \mu$ and $\nu$.

Since Theorem~\ref{GS_hpd:thm} deals with the hpd case, we now assume that there is no convection, $\mu = \nu = 0$. Then $B$ is the discretization of the diffusive term using central finite differences and as such it is an irreducible diagonally dominant M-matrix and thus hpd; see, e.g., \cite{BermanPlemmons}.
Accordingly, $A = I + \frac{\tau}{2}B$ is hpd as well. We took $N = 100$ which gives a spacing of $h=\tfrac{1}{N+1}$, and $\tau = 0.5h^2$ and 
we consider two cases:
constant diffusion coefficients 
\begin{equation} \label{const_diffusion:eq}
\alpha(x,y) = \beta(x,y) = 1,
\end{equation}
which gives a constant diagonal in $A$, 
and 
non-constant diffusion coefficients
\begin{equation} \label{var_diffusion:eq}
\alpha(x,y) = \beta(x,y) = 1+9 (x+y) ,
\end{equation}
which makes the entries on the diagonal of $A = I + \frac{\tau}{2}B$ vary between $1+\frac{4\tau}{2h^2}$ 
and $1+\frac{38\tau}{2h^2}$, i.e., between 2 and 9.5. 

Figure~\ref{hpd:fig} reports the numerical results. We chose the right hand side $\frac{\tau}{2} B c(t)$ 
in~\eqref{the_system:eq} as $Az$, where $z$ is the discretized evaluation of the function $xy(1-x)(1-y)$. 
So we know the exact solution, which allows us to report $A$-norms of the error, which is what we 
provided bounds for in Theorem~\ref{GS_hpd:thm}. The figure displays these $A$-norms only after every $n
$ relaxations, which we treat as one ``iteration'', since $n$ relaxations indeed make make up one iteration in standard Gauss-Seidel.

\begin{figure}
\centerline{\includegraphics[width=0.5\textwidth]{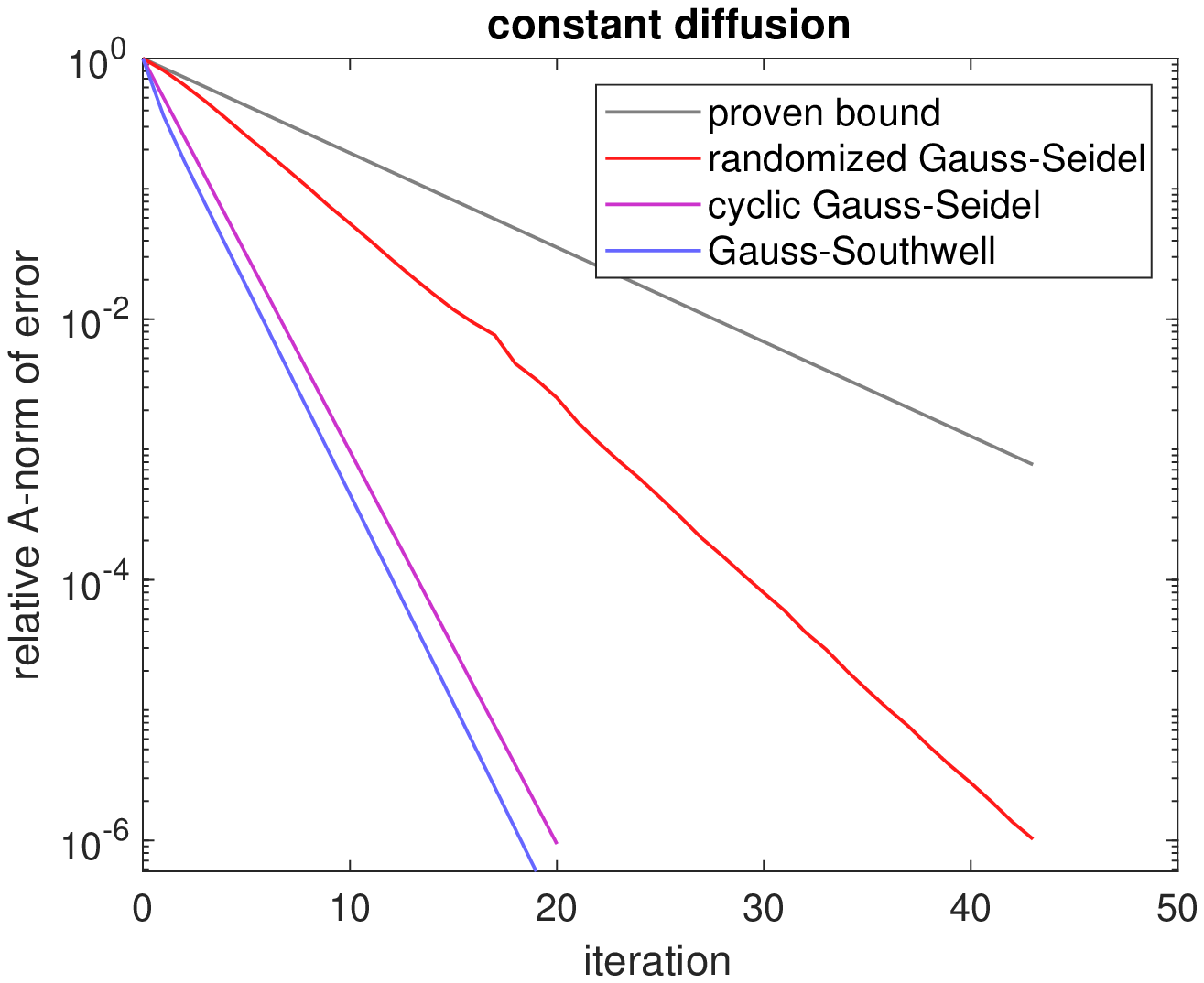}
\includegraphics[width=0.5\textwidth]{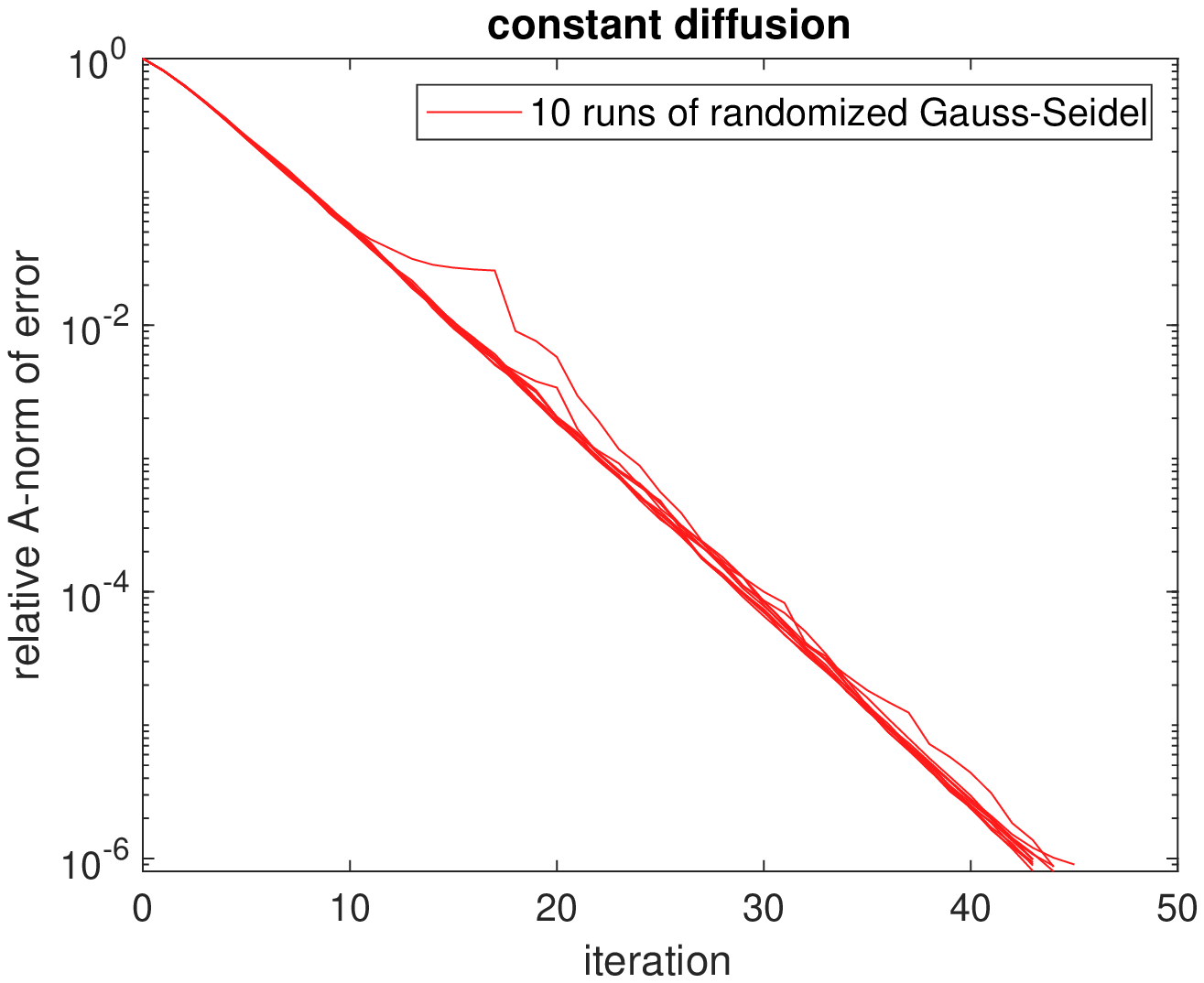}
}
\centerline{\includegraphics[width=0.5\textwidth]{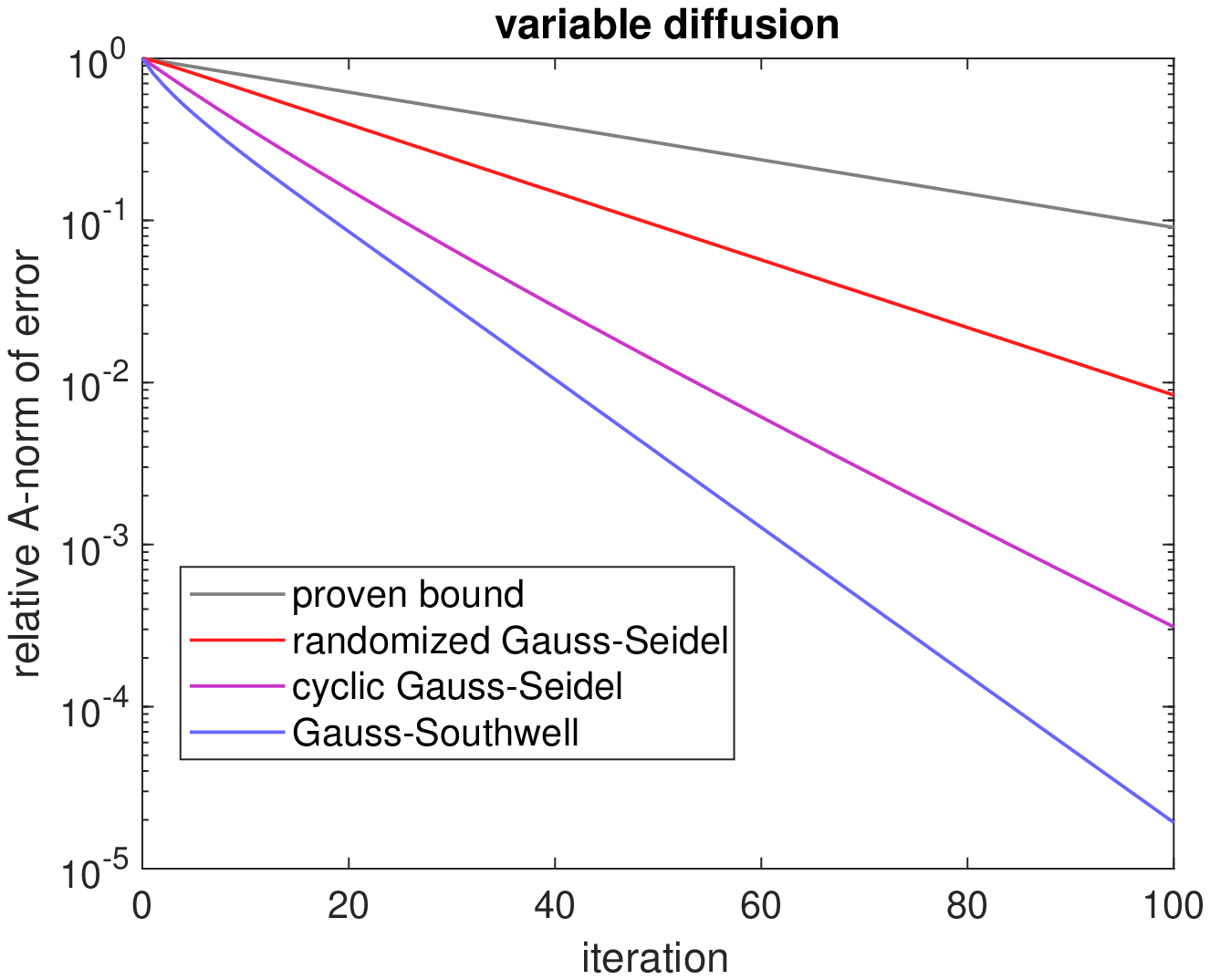}
\includegraphics[width=0.5\textwidth]{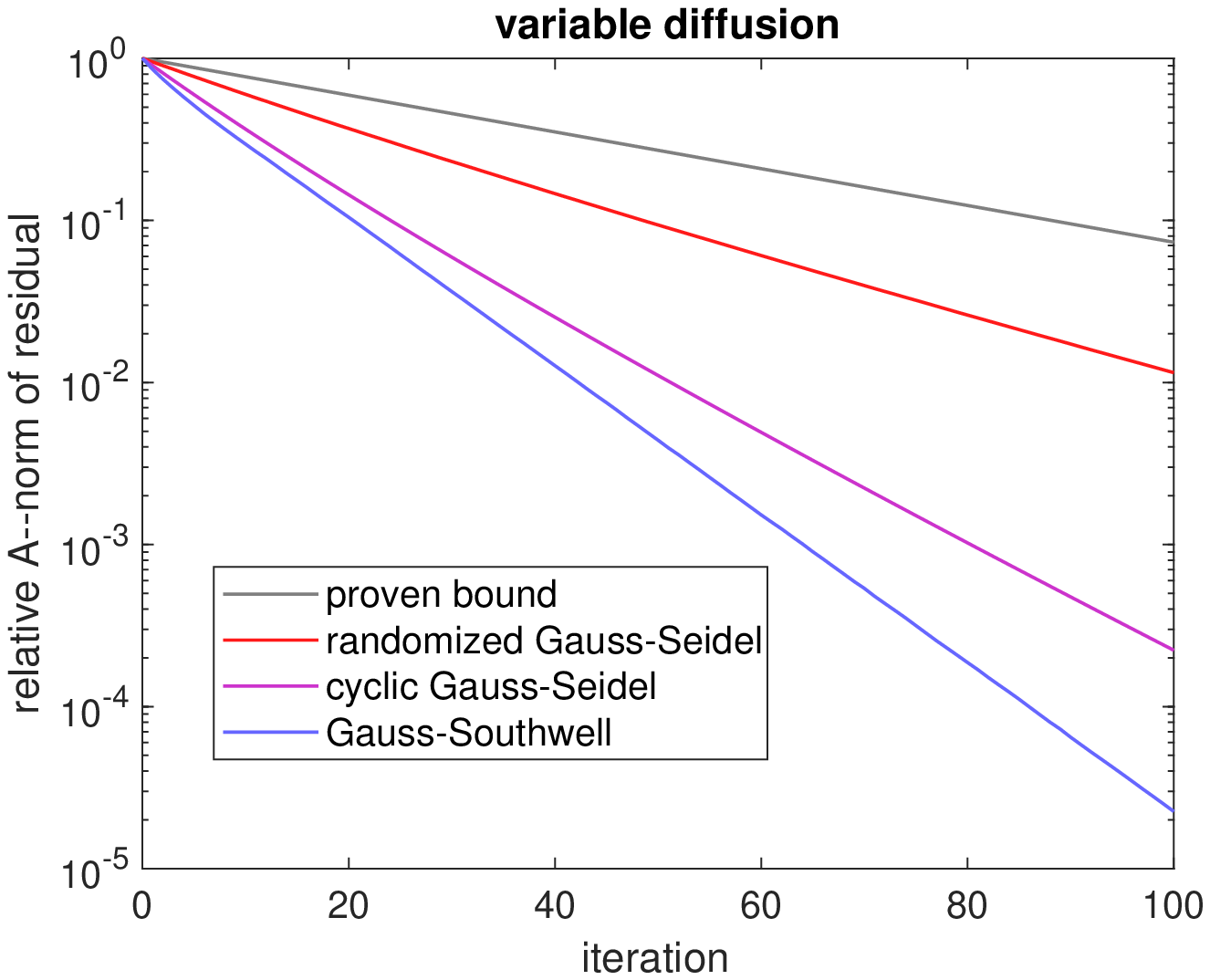}
}
\caption{Gauss-Seidel, randomized Gauss-Seidel and Gauss-Southwell. \textit{Top:} constant diffusion \eqref{const_diffusion:eq}. \textit{Bottom:} variable diffusion \eqref{var_diffusion:eq}. See text for the other parameters used. Top-down order in the legend corresponds to top-down order of the {plotted lines}. \label{hpd:fig}}
\end{figure}

The top row of Figure~\ref{hpd:fig} gives results for the constant diffusion case \eqref{const_diffusion:eq}. The left diagram shows the relative $A$-norm of the errors for randomized Gauss-Seidel with uniform probabilities $p_i = 1/n$, standard (``cyclic'') Gauss-Seidel and Gauss-Southwell with the greedy pick rule \eqref{standard_pick:eq}. 
For randomized Gauss-Seidel we actually give here---as in all other experiments---the averages %of the $A$-norms 
for ten runs which we regard as an approximation to the expected values. 
The plot to the right shows that the convergence behavior of these ten different runs exhibit only mild deviations. The plot on the left also contains the bound $(1- \alpha_{\opt})^{k/2}$ of Theorem~\ref{GS_hpd:thm}. We see that randomized Gauss-Seidel converges approximately half as fast as cyclic Gauss-Seidel, that Gauss-Southwell converges somewhat faster than cyclic Gauss-Seidel, and that the thoretical bounds are not very tight.

The bottom row of Figure~\ref{hpd:fig} shows results for variable diffusion according to 
\eqref{var_diffusion:eq}. Convergence is slower than in the constant diffusion case. The left plot has the results for the ``optimial'' probabilities  $p_i = \frac{a_{ii}}{\trace(A)}$ and the ``optimal'' greedy-pick \eqref{hpd_greedy_opt:eq}, for which the bounds of Theorem~\ref{GS_hpd:thm} hold again and are also reported, whereas the right plot shows the results for the uniform probabilities $p_i = 1/n$ and 
the greedy-pick rule \eqref{standard_pick:eq}. 
\af{In this case, the bound of 
Theorem~\ref{GS_hpd:thm} holds with $\alpha^{\rGS} = 
\alpha^{\GSW} = (\lambda_{\min}/n) \min_{i=1} 
1/a_{ii}$, and this bound is also plotted.} 
Interestingly, the two plots are virtually 
indistinguishable, except for a tiny improvement of 
randomized Gauss-Seidel when using ``optimal'' 
probabilities. We conclude that the choice of 
probabilities or the greedy pick rule has only a 
very marginal effect in this example. 
\af{The plots also show that the proven bounds 
can be pessimistic in the sense that the actual convergence is significantly faster. This is not uncommon when dealing with randomized algorithms, and we will address this further when discussing the numerical results illustrating the new convergence theorems for diagonally dominant matrices in Section~\ref{random_h:sec}.}

Although not being further addressed in this paper, we now shortly present basic numerical results on the performance 
of the various relaxation methods when used as a smoother in a multigrid method. This was mentioned as a possible application in the introduction. We consider a V-cycle multigrid method for the standard discrete 
Laplacian on a $N \times N$ grid (with $N{+1}$ a power of 2) with Dirichlet boundary conditions. Restriction and prolongation are done via the usual linear interpolation, doubling the grid spacing from one level to the next and going down to a minimum 
grid size of $7 \times 7$; see \cite{Trottenberg}. For each smoother, on each level $\ell$ with a grid size of $N_\ell \times N_\ell$ we perform a constant number of $s$ post- and $s$ pre-smoothing ``iterations'' amounting to $s N_\ell$ relaxations. For standard Gauss-Seidel and Gauss-Southwell we use $s = 1$, whereas for randomized Gauss-Seidel we tested 
$s =1, s = 1.5$ and $s = 2$. 

\begin{figure}
\centerline{\includegraphics[width=0.5\textwidth]{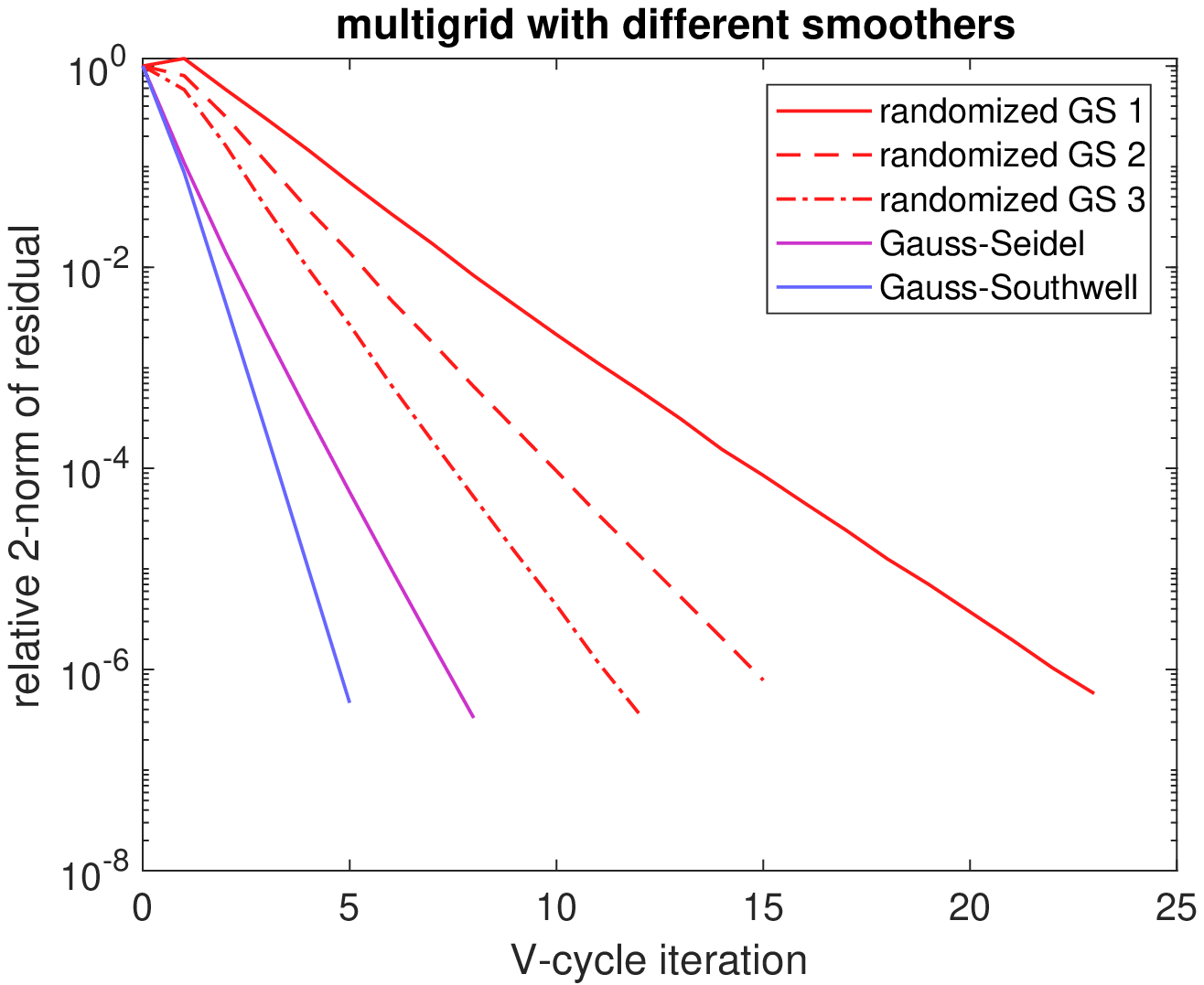}
\includegraphics[width=0.5\textwidth]{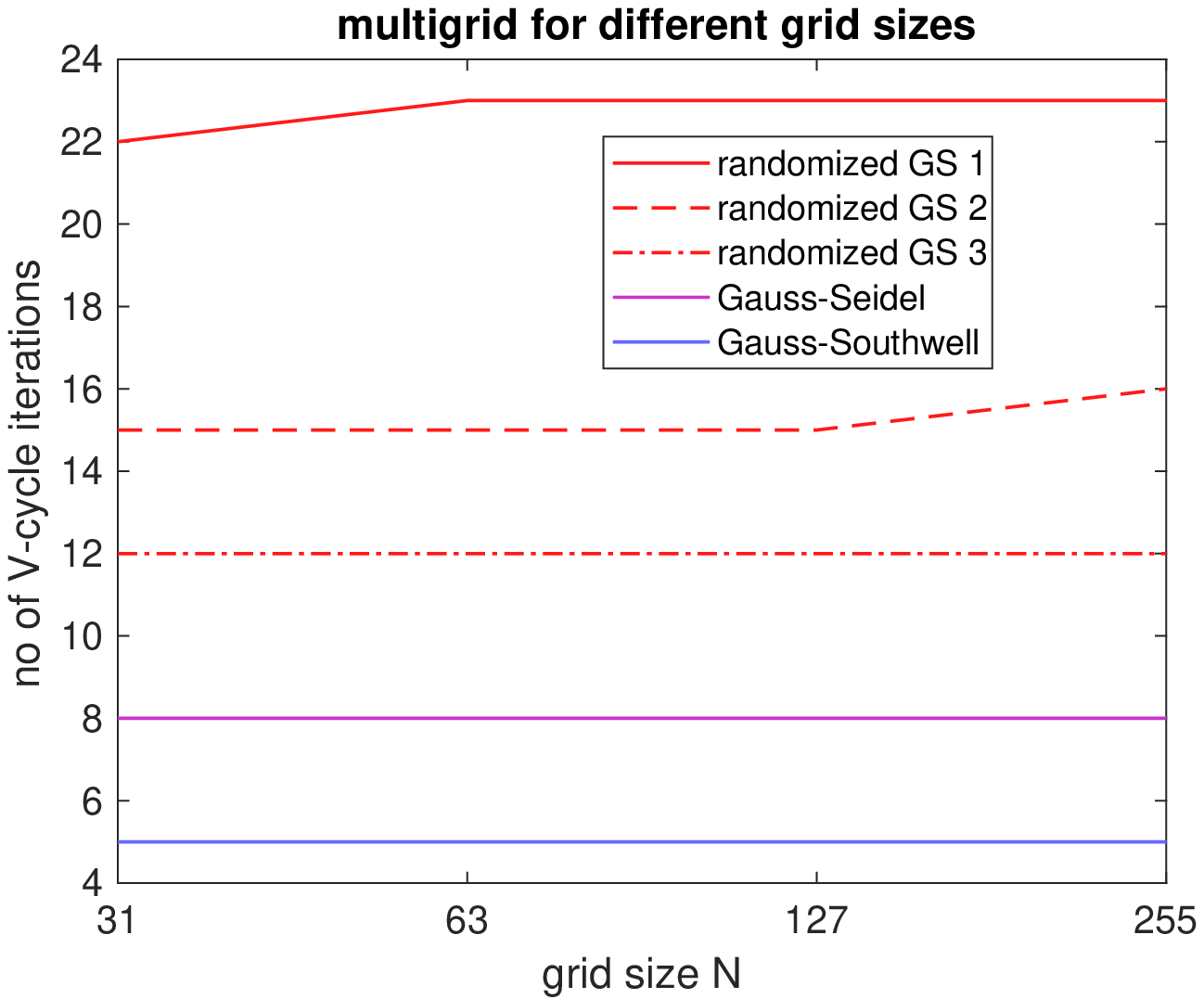}
}
\caption{Different relaxation schemes as smoothers in multigrid. Randomized GS 1, 2 and 3 correspond to $s =1$, $s = 1.5$ and $s = 2$ smoothing ``iterations'' in randomized Gauss-Seidel. \textit{Left:} relative residual norms for $N=127$. \textit{Right:} no.\ of V-cycles to reduce the initial residual by a factor of $10^{-6}$ for different grid sizes. 
\label{mg:fig}}
\end{figure}

Figure~\ref{mg:fig} indicates that randomized Gauss-Seidel has its potential for being used 
as a smoother in multigrid, provided that its slower convergence can be outweighed by a more efficient implementation, as it might be possible in a parallel environment. The left plot gives convergence plots for $N=127$. 
The right plot reports the number of V-cycles required to reduce the initial residual by a factor of $10^{-6}$ for various grid sizes $N$. We see that the convergence speed with randomized Gauss-Seidel and Gauss-Southwell smoothing is independent of the grid size just as with standard Gauss-Seidel smoothing, thus preserving one of the most important properties of the multigrid approach.  An interesting 
feature that random Gauss-Seidel shares with the Gauss-Southwell method is that we can 
prescribe a fractional number of smoothing iterations and thus adapt the computational work 
on a finer scale than with standard Gauss-Seidel. The figure also shows that for this 
example Gauss-Southwell yields faster convergence than standard Gauss-Seidel for the same 
number of relaxations.      

\section{Results for non-Hermitian matrices} \label{random_h:sec}
We now present several theorems which are counterparts to
Theorem~\ref{GS_hpd:thm} for classes of not necessarily Hermitian matrices, and
the iteration matrix $H=M^{-1}N$ in \eqref{generic_linear:eq} may arise from a general splitting $A = M-N$ other than the (relaxed) Jacobi splitting.  
In place of the $A$-norm we will now use weighted $\ell_1$-norms.

\begin{definition} \label{weighted_l1:def} For a given vector $u \in \mathR^n$ with positive components $u_i >0$, 
$i=1,\ldots,n$, the weighted $\ell_1$-norm on $\mathC^n$ is defined as
\[
\| x\|_{u,1} = \sum_{j=1}^n u_j |x_j| .
\]
\end{definition}

Clearly, the standard $\ell_1$-norm is obtained for $u=(1,\ldots,1)^T$. It is easy to see that the associated operator norm for $A \in \mathC^{n \times n}$ is the weighted column sum-norm
\[
\| A \|_{u,1} = \max_{j=1}^n  \frac{1}{u_j}  \sum_{i=1}^n u_i |a_{ij}|   .
\]

In the theorems to follow we will state results in terms of the preconditioned residual 
\[
\hat{r}^k = c - (I-H)x^k = M^{-1}(b-Ax^k) = M^{-1}A(x^*-x^k),
\]
and we denote $K$ the preconditioned matrix $K = M^{-1}A = I-H$.

Our first theorem assumes $\|H\|_{u,1} < 1$ and gives bounds on  the weighted $\ell_1$-norms of the preconditioned residuals in Algorithms~\ref{greedy_linear:alg} and \ref{randomized_linear:alg} similar in nature 
to those in Theorem~\ref{GS_hpd:thm}. 

\begin{theorem}\label{H_conv:thm} Consider the weighted column sums 
\[
\rho_j = \frac{1}{u_j} \sum_{i=1}^n u_i |h_{ij}|, \enspace j=1,\ldots,n,
\]
and assume that $\|H\|_{1,u} = \max_{j=1}^n \rho_j < 1$. Set $\gamma_j := (1-\rho_j)^{-1}, j=1,\ldots,n$.
\begin{itemize}
    \item[(i)]  In randomized relaxation (Algorithm~\ref{randomized_linear:alg}), putting  %Assume that in the randomized Algorithm~\ref{randomized_linear:alg} all probabilities $p_i$ are positive and let 
\begin{equation} \label{alpha_l1:eq}
\alpha^\r = \min_{j=1}^n \frac{p_j}{\gamma_j} ,
\end{equation}
%Let $K$ denote the non-singular matrix $K=I-H$.
 the expected values for the weighted $\ell_1$-norm of the preconditioned residuals $\hat r^k= M^{-1}r^k = c-Kx^k$ of the iterates $x^k$  satisfy
\begin{equation} \label{randomized_H_conv:eq}
\mathbb{E}(\|\hat r^k\|_{1,u}) \leq \left(1-\alpha^\r \right)^{k} \| \hat r^0  \|_{1,u}.
\end{equation}
The quantity $\alpha^\r$ in \eqref{alpha_l1:eq} is maximized if one takes 
\begin{equation} \label{optimal_p:eq}
p_i = \gamma_i/\sum_{j=1}^n \gamma_j, \enspace j=1,\ldots,n; 
\end{equation}
its value then is $\alpha^\r = \alpha_{\opt} := 1 /  \sum_{j=1}^n \gamma_j$.

\item[(ii)] In greedy relaxation (Algorithm~\ref{greedy_linear:alg}), with the greedy pick rule \eqref{prec_pick:eq} based on the preconditioned residual, putting
\[
\alpha^\g =  \min_{j=1}^n \frac{\pi_j}{\gamma_j}, \mbox{ where } \pi_j = \frac{u_j/\beta_j}{\sum_{\ell=1}^n u_\ell/\beta_\ell},
\]
the weighted $\ell_1$-norms of the preconditioned residuals $\hat r^k= M^{-1}r^k = c-Kx^k$ of the iterates $x^k$ satisfy
\begin{equation} \label{greedy_H_conv:eq}
 \|\hat r^k\|_{1,u} \leq \left(1-\alpha^\g \right)^{k} \| \hat r^0  \|_{1,u}.
\end{equation}
Moreover, $\alpha^\g$ is maximized if we take 
\begin{equation} \label{optimal_beta:eq}
\beta_j = u_j/\gamma_j,  \enspace j=1,\ldots,n; 
\end{equation}
its maximal value is identical to $\alpha_\opt$ from part (i).
\end{itemize}
\end{theorem}
\begin{proof} If $i$ is the index chosen at iteration $k$, we have
\[
x_{i}^{k+1} = \sum_{j=1}^n h_{ij} x_i^{k} + c_i = x_i^{k} + \hat r^k_i,
\]
which gives
\[
x^{k+1} = x^{k} + \hat r^k_i e_i, \enspace 
\hat r^{k+1} = c - Kx^{k+1} = \hat r^{k}- \hat r^k_iKe_i.
\]
We therefore have
\begin{eqnarray}
\|\hat r^{k+1}\|_{1,u} &=& \sum_{\ell=1,\ell \neq i}^n {u_\ell}|\hat r_\ell^{k} + \hat r^k_i h_{\ell i}| + {u_i} |\hat r^k_i h_{ii}| \nonumber \\
&\leq&  \sum_{\ell=1,\ell \neq i}^n {u_\ell} |\hat r_\ell^{k}| + |\hat r^k_i| \sum_{\ell=1}^n {u_\ell}|h_{\ell i}| \nonumber \\
&=&  \sum_{\ell=1,\ell \neq i}^n {u_\ell} |\hat r_\ell^{k}| + {u_i}|\hat r^k_i| \cdot \frac{1}{u_i} \sum_{\ell=1}^n {u_\ell}|h_{\ell i}| \nonumber \\
&=&  \|\hat r^{k}\|_{1,u} - (1-\rho_i){u_i}|\hat r^k_i|. \label{basic_inequality:eq}
\end{eqnarray}
To prove part (i) we use \eqref{basic_inequality:eq} to see that the expected value of the norm of the residual 
$\hat r^{k+1}$, conditioned to the given value for $\hat r^k$ satisfies 
\[
\mathbb{E} \left(\|\hat r^{k+1}\|_{1,u} \, \big| \, \hat r^k \right) \leq \sum_{i=1}^n p_i \left( \| \hat r^{k}\|_{1,u} - (1-\rho_i){u_i}|\hat r^k_i|\right) \leq (1-\alpha) \|\hat r^k\|_{1,u},
\]
from which we get \eqref{randomized_H_conv:eq}. Moreover, the minimum $\alpha^\r = \min_{j=1}^n p_j/\gamma_j$ is not larger than the 
convex combination $\sum_{j=1}^n \left(\gamma_j/\sum_{\ell=1}^n \gamma_\ell\right) \cdot  p_j/\gamma_j = 1/ \sum_{\ell=1}^n \gamma_\ell$, 
and this value is attained for $\alpha^\r$ if we choose $p_i = \gamma_i/\sum_{\ell=1}^n \gamma_\ell$. 

To prove part (ii) we observe that due to the greedy pick rule \eqref{prec_pick:eq} we have
\begin{eqnarray*}
\frac{u_i}{\gamma_i}|r_i^k| &=& \frac{u_i}{\gamma_i \beta_i} \max_{j=1}^n \frac{\beta_j}{u_j} u_j |r^k_j|,
\end{eqnarray*}
which, using Lemma~\ref{convex_comb:lem}, gives
\begin{eqnarray*}
\frac{u_i}{\gamma_i}|r_i^k| & \geq & \frac{u_i}{\gamma_i \beta_i} \frac{1}{\sum_{\ell=1}^n \frac{u_\ell}{\beta_\ell}} \sum_{j=1}^n u_j |r_j^k| \\
&= &\frac{u_i}{\gamma_i \beta_i} \frac{1}{\sum_{\ell=1}^n \frac{u_\ell}{\beta_\ell}} \|r^k\|_{1,u} \\
& \geq & \min_{j=1}^n \frac{u_j}{\gamma_i \beta_j} \frac{1}{\sum_{\ell=1}^n \frac{u_\ell}{\beta_\ell}} \|r^k\|_{1,u}. 
\end{eqnarray*}
Together with \eqref{basic_inequality:eq} this gives \eqref{greedy_H_conv:eq}. Finally, using Lemma~\ref{convex_comb:lem} again, we obtain
\[
  \min_{j=1}^n \frac{u_j}{\gamma_j\beta_j}  \leq \frac{1}{\sum_{j=1}^n \gamma_j} \sum_{j=1}^n \frac{u_j}{\beta_j} ~,
\]
which gives
\[
 \alpha^\g = \min_{j=1}^n \frac{u_j}{\gamma_j\beta_j} \frac{1}{\sum_{\ell=1}^n \frac{u_\ell}{\beta_\ell}}
\leq \frac{1}{\sum_{\ell=1}^n \gamma_\ell} = \alpha_\opt ~.
\]
And $\alpha_\opt$ is attained as value for $\alpha^\g$ if we take $\beta_j = u_j/\gamma_j$ for $j=1,\ldots,n$.
\end{proof}

\af{The convergence results of Theorem~\ref{H_conv:thm}
are given in terms of the weighted $\ell_1$-norm, since it is 
this norm for which we can prove a decrease in every relaxation 
due to the assumption $\|H\|_{u,1} < 1$. As we will soon see, for 
randomized Gauss-Seidel and Gauss-Southwell this assumption is 
equivalent to a (generalized) diagonal dominance assumption on 
$A$, a condition which is often fulfilled in applications and 
which can be checked easily, at least when the weights are all~1. 
In this context, it is worth mentioning that results like the bound \eqref{greedy_H_conv:eq} yield a bound on the $R_1$ convergence factor of the sequence $x^k-x^*$, the standard measure of the convergence rate for a linearly zero-convergent sequence defined as
\[
R_1(\{x^k-x\}) = \limsup \|x^k-x^*\|^{1/k};
\]
see, e.g., \cite{OrtegaRheinboldt1970}.
The $R_1$-factor is independent of the norm $\| \cdot 
\|$, and results like \eqref{greedy_H_conv:eq} may be 
interpreted in a norm-independent manner by saying 
that 
\[
R_1(\{x^k-x\}) \leq (1-\alpha^{\rGS}). 
\]
}

From Theorem~\ref{H_conv:thm} we see that with the 
optimal choices for the probabilities $p_i$ or the 
weights $\beta_i$, the proven bounds for randomized 
relaxation and greedy relaxation are identical. \af{So from the 
point of view of the established theory we cannot conclude that 
randomized would outperform greedy or vice-versa. In all our 
practical experiments, though, the greedy approach exposed 
faster convergence than the randomized approach.}
%Assuming that these bounds actually reflect the 
%convergence behavior observed in a computation, we 
%may conclude that randomized relaxation can be 
%expected to converge as fast as greedy relaxation 
%with the advantage of not having to compute maxima 
%of the residuals in each relaxation. 

Also note that if we just take $p_i = 1/n$ for all~$i$ in randomized relaxation, then
\[
\alpha^\r = \min_{j=1}^n \frac{1}{n \gamma_j} = \frac{1}{n} \left( 1-\max_{j=1}^n \rho_j\right),
\]
and the same value is attained for $\alpha^\g$ in greedy relaxation if we take  $\beta_i = u_i$ for all~$i$. \af{The optimal value $\alpha_{\opt}$ is attained for the greedy pick rule \eqref{optimal_beta:eq}. If we take the standard greedy pick rule \eqref{standard_pick:eq}, i.e.\ } $\beta_i = 1$ for all $i$, we have
\[
\alpha^\g = \min_{j=1}^n \frac{u_j}{\gamma_j} \cdot \frac{1}{ \sum_{\ell =1}^n{u_\ell}},
\]
which, depending on the values of $u_j$ can be smaller or larger than 
\linebreak
$\tfrac{1}{n}(1-\max_{j=1}^n \rho_j)$ \af{but is certainly never larger than $\alpha_{\opt}$ obtained with the pick rule \eqref{optimal_beta:eq}.}  

In Theorem~\ref{H_conv:thm} we need to know the weights $u_i$ and with them the weighted column sums $\rho_i$ in order to be able to choose the probabilities $p_i$ or the greedy pick for which we get the strongest convergence bound, 
i.e., the largest value for~$\alpha^\r$ and~$\alpha^\g$. 
For example, it might be that we can take $u = (1,\ldots,1)$, such that $\| \cdot \|_{1,u}$ reduces to the standard $\ell_1$-norm. %In such a situation we can caculate the values $\rho_j$ explicitly, and then use the optimal values for the probabilites $p_j$ given by Theorem~\ref{H_conv:thm}. 
However, it might also be that we know that $\|H\|_{1,u} < 1$ for some $u>0$ without knowing $u$ explicitly. Theorem~\ref{H_conv:thm} tells us that we still have convergence for any choice of probabilities $p_i$ in randomized relaxation or weights $\beta_i$ in greedy relaxation, but the proven convergence bounds are weaker than for the ``optimal'' probabilites \eqref{optimal_p:eq} or weights \eqref{optimal_beta:eq}.

In light of this discussion it is interesting that for a particular vector of weights $u$ we can somehow reverse the situation, at least for the randomized iteration: We know how to choose the corresponding optimal values for the probabilities while we do not need to know $u$ explicitly. 

In order to prepare this result we recall the following left eigenvector version of the Perron-Frobenius 
theorem; see, e.g., \cite{varga-62}. 
Note that a square matrix $H$ is called irreducible if there is no permutation matrix $P$ such that $P^TAP$ has a $2\times 2$ block structure with a zero off-diagonal block. We also use the notation $H \geq 0$ (``$H$ is nonnegative'') if all entries $h_{ij}$ of $H$ are nonnegative. Similarly, a vector $w \in \mathR^n$ is called nonnegative ($w \geq 0$) or positive ($w > 0$), if all its components are nonnegative or positive, respectively.    

\begin{theorem} \label{Perron_Frobenius:thm} 
Let $H \in \mathR^{n \times n}$, $H \geq 0$, be irreducible. Then there exists a positive vector $w \in \mathR^n$, the ``left Perron vector'' of $H$, such that $w^T H = \rho(H) w^T$, where $\rho(H)$ is the spectral radius of $H$. Moreover, $w$ is unique up to scaling with a positive scalar.
\end{theorem}

%We call $w$ a ``left Perron vector'' fpr $H$.

A direct consequence of Theorem~\ref{Perron_Frobenius:thm} is that for $H\geq 0$ irreducible we have 
\[
\rho_j = \frac{1}{w_j} \sum_{i=1}^n w_i h_{ij} = \rho \mbox{ for } j=1,\ldots,n,
\]
and thus $\| H \|_{w,1} = \rho$.  

If $H \geq 0$ is not irreducible, a positive left Perron vector needs not necessarily exist. However, we have the following approximation result which we state as a lemma for future reference.

\begin{lemma} \label{approx_Perron:lem} Assume that $H \in \mathR^{n \times n}$ is nonnegative. Then, for any $\epsilon > 0$ there exists a positive vector $w_\epsilon > 0$ such that $w_\epsilon^T H \leq (\rho+\epsilon)w_\epsilon^T$.
\end{lemma}
\begin{proof} For given $\epsilon > 0$, due to the continuity of the spectral radius we can choose $\delta >0$ small enough such that the spectral radius of the irreducible matrix $H_\delta = H + \delta E$, $E$ the matrix of all ones, is less or equal than $\rho +\epsilon$. Now take $w_\epsilon$ as the left Perron vector of $H_\delta$.
\end{proof}

We are now ready to prove the following theorem where we use the notation $|H|$ for the matrix resulting from $H$ when replacing each entry by its absolute value. \af{Interestingly, the theorem establishes a situation where we know how to choose optimal probabilities (in the sense of the proven bounds) without explicit knowledge of the weights.}

\begin{theorem}\label{randomized_H_conv_w:thm} Assume that $\rho= \rho(|H|) < 1$ and consider randomized relaxation (Algorithm~\ref{randomized_linear:alg}).  
\begin{itemize} 
\item[(i)] If $H$ is irreducible, then there exists a positive vector of weights $w$ such that the 
weighted $\ell_1$-norm of the preconditioned residuals $\hat r^k= M^{-1}r^k$ of the iterates $x^k$ satisfies
\begin{equation*} \label{randomized_H_conv_w_irreducible:eq}
\mathbb{E}(\|\hat r^k\|_{1,w}) \leq \left(1-\alpha \right)^{k} \| \hat r^0  \|_{1,w}, \mbox{ where } \alpha = (1-\rho)\min_{j=1}^n p_j.
\end{equation*}
Moreover, $\alpha$ is maximized if one takes $p_j = 1/n$ for $j=1,\ldots,n$; its value then is 
$\alpha_{\opt} = (1-\rho)/{n}$. 
\item[(ii)] If $H$ is not irreducible, then for every $\epsilon > 0$ such that $\rho + \epsilon < 1$ there exists a positive vector of weights $w_\epsilon$ such that the weighted $\ell_1$-norm of the preconditioned residuals $\hat r^k= M^{-1}r^k$ of the iterates $x^k$  satisfies
\begin{equation*} \label{randomized_H_conv_w_reducible:eq}
\mathbb{E}(\|\hat r^k\|_{1,w_\epsilon}) \leq \left(1-\alpha^\epsilon \right)^{k} \| \hat r^0  \|_{1,w} \mbox{ where } \alpha^\epsilon = (1-(\rho+\epsilon))\min_{j=1}^n p_j
\end{equation*}
Moreover, $\alpha^\epsilon$ is maximized if one takes $p_j = 1/n$ for $j=1,\ldots,n$; its value then 
%is $\alpha^\epsilon_{\opt} = \tfrac{1-\rho-\epsilon}{n}$.
is $\alpha^\epsilon_{\opt} = (1-\rho-\epsilon)/{n}$.
\end{itemize}
\end{theorem}
\begin{proof} 
Part (i) follows immediately from Theorem~\ref{H_conv:thm} by taking $w$ as the left Perron vector 
of $|H|$, noting that with this vector we have $\rho_j = \rho$ for $j=1,\ldots,n$. Part (ii) follows from Theorem~\ref{H_conv:thm}, too, now taking $w_\epsilon$ as the vector from Lemma~\ref{approx_Perron:lem}, observing that for this vector we have $\rho_j \leq \rho + \epsilon$ 
for $j=1,\ldots,n$ for the weighted column sums 
$$\rho_j = \frac{1}{(w_\epsilon)_j} \sum_{i=1}^n (w_\epsilon)_i |h_{ij}|. %~~ \cvd 
$$
\end{proof}

Interestingly, Theorem~\ref{randomized_H_conv_w:thm} cannot be transferred to greedy relaxation, at least not with the techniques used there. Indeed, in order to obtain a bound
$\| \hat r^{k+1}\|_{1,w} \leq (1-\alpha_{\opt})^k \| \hat r^k\|_{1,w}$ when $H$ is irreducible, e.g., the bounds given in Theorem~\ref{H_conv:thm} tell us that we would have to use the greedy pick rule
\[
{w_i}|\hat r^k_i| = \max_{j=1}^n {w_j}|\hat r^k_j|,
\]
which requires the knowledge of $w$.

\section{Randomized Gauss-Seidel and Gauss-Southwell for H-matrices}
Building on Theorem~\ref{H_conv:thm} we now derive convergence results for the randomized Gauss-Seidel and the Gauss-Southwell method when $A$ is an H-matrix.

\begin{definition} \label{H_M:def} (See, e.g., \cite{BermanPlemmons}) 
\begin{itemize}
\item[(i)]
A matrix $A =(a_{ij}) \in \mathR^{n \times n}$ is called an {\em M-matrix} if $a_{ij} \leq 0$ for $i \neq j$ and it is non-singular with $A^{-1} \geq 0$.
\item[(ii)]
A matrix $A \in \mathC^{n \times n}$ is called an {\em H-matrix}, if its comparison matrix $\langle A \rangle$ with
\[
\langle A \rangle_{ij} = \left\{ \begin{array}{rl} |a_{ii}| & \mbox{if $i = j$} \\
                                                 -|a_{ij}| & \mbox{if $i \neq j$}
                                \end{array}
                                \right.
\]
is an M-matrix.
\end{itemize}
\end{definition}

Clearly, an M-matrix is also an H-matrix. For our purposes it is important that 
H-matrices can equivalently be characterized as being generalized diagonally dominant.

\begin{lemma} \label{vectors_for_H:lem}Let $A \in \mathC^{n \times n}$ be an H-matrix. Then
\begin{itemize}
\item[(i)] There exists a positive vector $v\in \mathR^n$  such that $A$ is generalized diagonally dominant by rows, i.e.,
\[
|a_{ii}|v_i > \sum_{j =1, j\neq i}^n |a_{ij}|v_j \mbox{ for } i = 1,\ldots,n.
\]
\item[(ii)] There exists a positive vector $u > 0$ such that $A$ is generalized diagonally dominant by columns, i.e.
\[
u_j |a_{jj}|  > \sum_{i =1, i\neq j}^n u_i |a_{ij}| \mbox{ for } j = 1,\ldots,n.
\]
\end{itemize}
\end{lemma}
\begin{proof} Part (i) can be found in many text books, one can take $v = \langle A  \rangle^{-1}e$, with $e= (1,\ldots,1)^T$.
Part (ii) follows similarly by taking $u^T$ as the row vector $e^T \langle A \rangle^{-1}$.
\end{proof}

The lemma implies the following immediate corollary. %It uses the notation $|H|$ for the matrix resulting from $H$ when replacing each entry by its absolute value.

\begin{corollary} \label{Jacobi_absolute_conv:cor} 
Let $A$ be an H-matrix and let $A = D-B$ be its Jacobi splitting with $D$ the diagonal part of $A$.
Then the iteration matrix $|D^{-1}B|$ belonging to the Jacobi splitting $\langle A \rangle = |D|-|B|$ of
$\langle A \rangle$ satisfies $\|D^{-1}B|\|_{1,u} < 1$ with $u$ the vector from Lemma~\ref{vectors_for_H:lem}(ii).
\end{corollary}

With these preparations we easily obtain the following first theorem on the (unrelaxed) randomized Gauss-Seidel and Gauss-Southwell methods. We formulate it using the residuals $b-Ax^k$ of the original equation.

\begin{theorem} \label{GS_H_conv:thm} Let $A$ be an H-matrix and let $u$ be a positive vector such that $u^T\langle A \rangle > 0$. Let $A = D-B$ be the Jacobi splitting of $A$ and put $H = D^{-1}B$. Moreover, let $w=(w_,\ldots,w_n)$ with $w_j = u_j / |a_{jj}|, j=1,\ldots,n$. Then 
\begin{itemize} 
\item[(i)] All weighted column sums 
\begin{equation} \label{rhoj_def:eq}
   \rho_j = \frac{1}{u_j}\sum_{i=1}^n |h_{ij}| u_i
   \end{equation} 
 are less than 1. 
\item[(ii)] In the randomized Gauss-Seidel method, i.e.\ Algorithm~\ref{randomized_linear:alg} for $H=D^{-1}B$, the expected values of the $w$-weighted $\ell_1$-norm $\|r^k\|_{1,w}$ of the original residuals satisfy
\[
\mathbb{E}(\|r^k\|_{1,w}) \leq \left(1-\alpha^\rGS \right)^{k} \| r^0  \|_{1,w},
\]
where $\alpha^\rGS =  \min_{j=1}^n (p_j/\gamma_j) >0$ and $\gamma_j = (1-\rho_j)^{-1}$ for $j=1,\ldots,n$. The value of $\alpha^\rGS$ is maximized for the choice $p_i = \gamma_i / \sum_{\ell=1}^n \gamma_\ell$, and the resulting value for $\alpha^\rGS$ is $\alpha_{\opt} = 1/\sum_{\ell=1}^n \gamma_\ell$. 
\item[(iii)] In the Gauss-Southwell method (Algorithm~\ref{greedy_linear:alg} with $H = D^{-1}B$), using the greedy pick rule~\eqref{prec_pick:eq} based on the preconditioned residual, the $w$-weighted $\ell_1$-norm of the original residuals of the iterates $x^k$ satisfy
\begin{equation*} 
 \| r^k\|_{1,w} \leq \left(1-\alpha^\g \right)^{k} \|  r^0  \|_{1,w},
\end{equation*}
with
\[
\alpha^\g =  \min_{j=1}^n \frac{\pi_j}{\gamma_j}, \mbox{ where } \pi_j = \frac{u_j/\beta_j}{\sum_{\ell=1}^n u_\ell/\beta_\ell}.
\]
Moreover, $\alpha^\g$ is maximized if we take 
\begin{equation*} 
\beta_j = u_j/\gamma_j,  \enspace j=1,\ldots,n; 
\end{equation*}
its maximal value is then $\alpha_\opt$ from (i).
\end{itemize}
\end{theorem}
\begin{proof}
For (i) observe that we have $|a_{jj}|u_j > \sum_{i=1, i\neq j}^n |a_{ij}| u_i$ for $j=1,\ldots,n$ and thus, since $h_{ij} = a_{ij}/a_{ii}$ for $i \neq j$ and $h_{jj} = 0$,
\[
\rho_j = \frac{1}{u_j} \sum_{i=1}^n |h_{ij}| u_i = \frac{1}{u_j|a_{jj}|} \sum_{i=1, i \neq j}^n |a_{ij}| u_i < 1.
\]
Parts (ii) and (iii) now follow directly from Theorem~\ref{H_conv:thm}, observing that for the   residual
$\hat r^k = D^{-1}r^k$ we have $\| \hat r^k \|_{1,u} = \| r^k \|_{1,w}$.
\end{proof}

Note that for Gauss-Southwell the greedy pick rule \eqref{prec_pick:eq} with weights $\beta_i$ based on the preconditioned residual is equivalent to the greedy pick rule \eqref{weighted_pick:eq} based on the original residual with weights $\beta_i/a_{ii}$.

Instead of changing the weights from $u$ to $w$, it is also possible to obtain a bound for the $u$-weighted 
$\ell_1$-norm, where, in addition, the same optimal choice for the $p_j$ in the randomized Gauss-Seidel iteration 
yields the same $\alpha_{\opt}$ as that of Theorem~\ref{GS_H_conv:thm}, and similarly for the Gauss-Southwell iteration.

\begin{theorem} \label{GS_H_conv2:thm} Let $A$ be an H-matrix and let $u$ be a positive vector such that $u^T\langle A \rangle > 0$. Then 
\begin{itemize} 
\item[(i)] In the randomized Gauss-Seidel method the expected values of the $u$-weighted $\ell_1$-norm $\|r^k\|_{1,u}$ of the residuals satisfy
\[
\mathbb{E}(\|r^k\|_{1,u}) \leq \left(1-\alpha \right)^{k} \| r^0  \|_{1,u},
\]
where $\alpha =  \min_{j=1}^n ({p_j}/{\gamma_j}) >0$ and $\gamma_j = (1-\rho_j)^{-1}$, $\rho_j$ from \eqref{rhoj_def:eq}, for $j=1,\ldots,n$. The value of $\alpha$ is maximized for the choice 
\begin{equation} \label{opt_prob_u:eq}
p_i = \gamma_i / \sum_{\ell=1}^n \gamma_\ell,
\end{equation}
and the resulting value for $\alpha$ is $\alpha_{\opt} = 1/\sum_{\ell=1}^n \gamma_\ell$. 
\item[(ii)] In the Gauss-Southwell method, if we take the same greedy pick as in Theorem~\ref{GS_H_conv:thm}, i.e.,  
\begin{equation} \label{greedy_pick2:eq}
{(1-\rho_i)}\frac{u_i}{|a_{ii}|}{|r^k_i|} = \max_{j=1}^n {(1-\rho_j)}\frac{u_j}{|a_{jj}|}{|r^k_j|},
\end{equation}
then
\[
   \|r^k\|_{1,u} \leq \left(1- \alpha_{\opt} \right)^{k} \| r^0  \|_{1,u}.
\]
\end{itemize}
\end{theorem}
\begin{proof}
If $i$ is the index chosen in iteration $k$ in randomized Gauss-Seidel or Gauss-Southwell, we have from \eqref{JOR_residual_formulation:eq}
\[
r^{k+1} = b-Ax^{k+1} = r^k - \frac{r^k_i}{a_{ii}} Ae_i.
\]
This gives
\begin{eqnarray*}
\| r^{k+1}\|_{1,u} &=& \sum_{\ell=1,\ell \neq i}^n {u_\ell} \left| r_\ell^{k} -  \frac{r^k_i}{a_{ii}} a_{\ell i}\right|  \nonumber \\
&\leq&  \sum_{\ell=1,\ell \neq i}^n {u_\ell} | r_\ell^{k}| + \frac{| r^k_i|}{|a_{ii}|} \sum_{\ell=1, \ell \neq i}^n {u_\ell}|a_{\ell i}| \nonumber \\
&=&  \sum_{\ell=1,\ell \neq i}^n {u_\ell} | r_\ell^{k}| + {u_i} |r^k_i| \frac{1}{u_i|a_{ii}|} \sum_{\ell=1,\ell \neq i}^n {u_\ell}|a_{\ell i}| \nonumber \\
&=&  \| r^{k}\|_{1,u} - (1-\rho_i) {{u_i}}| r^k_i|. 
\end{eqnarray*}
This is exactly the same relation as \eqref{basic_inequality:eq}, but now for the original residuals rather than the preconditioned ones. Parts (i) and (ii) therefore follow exactly in the same manner as in the proof of Theorem~\ref{H_conv:thm}.
\end{proof}

Let us mention that, if $A$ and thus $|H|=|D^{-1}B|$ is irreducible, the left Perron vector $u$ of $|H|$ is a vector with $u^T\langle A \rangle > 0$. As was 
discussed after Theorem~\ref{H_conv:thm} this means that with respect to the weigths from this vector we know the 
optimal probabilities in randomized Gauss-Seidel to be $p_i = 1/n, i=1,\ldots,n$, i.e., we do not need to know $u$ 
explicitly. According to Lemma~\ref{approx_Perron:lem}, a similar observation holds in an approximate sense with arbitrary 
precision when $A$ is not irreducible. 

{We also remark that the preconditioned residual $\hat r^k = D^{-1}r^k$ satisfies $\|\hat r^k\|_{1,u} = \| r^k \|_{1,w}$ with $u,w$ from Theorems~\ref{GS_H_conv:thm} and \ref{GS_H_conv2:thm}. So with these two theorems we 
have obtained identical convergence bounds for the $u$-weighted $\ell_1$-norm of the unpreconditioned and the preconditioned residuals.}

Theorem~\ref{GS_H_conv:thm} can be extended to relaxed randomized Gauss-Seidel iterations. We formulate the results only for the case where the weight vector is the left Perron vector. While this is not mandatory as long as the relaxation parameter $\omega$ satisfies $\omega \in (0,1]$, it is crucial for the part which extends the range of $\omega$ to values larger than~1. 

\begin{theorem} \label{relaxed_rGS_conv_th} Let $A$ be an H-matrix and let $A = D-B$ be its Jacobi splitting.  Put $H = D^{-1}B$ and $\rho = \rho(|H|) <1$. Assume that $\omega \in (0,\tfrac{2}{1+\rho})$ and define 
\linebreak
$\rho_\omega = \omega \rho + |1-\omega| \in (0,1)$. Consider the relaxed randomized Gauss-Seidel iteration, i.e., 
Algorithm~\ref{randomized_linear:alg} with the matrix $H_\omega$ from \eqref{JOR_matrix:eq}. Then 
\begin{itemize}
\item[(i)]
If $A$ and thus $|H|$ is irreducible, then with $u$ the left Perron vector of $|H|$ and $w$ the positive vector with components $w_i = u_i / |a_{ii}|$, the expected values for the $w$-weighted $\ell_1$-norm of the residuals satisfy
\[
\mathbb{E}(\|r^k\|_{1,w}) \leq \left(1- \alpha_\omega \right)^{k} \| r^0  \|_{1,w},
\]
where $\alpha_\omega =  \min_{j=1}^n p_j/\gamma_\omega >0$, $\gamma_\omega = (1-\rho_\omega)^{-1}$. The value of $\alpha_\omega$ is maximized for the choice $p_i = {1}/{n}$, and the resulting value for 
$\alpha_\omega$ is $\alpha_\omega^{\opt} = (1-\rho_\omega)/{n}$.
\item[(ii)] If $A$ is not irreducible, then for every $\epsilon > 0$ such that $\rho_\omega + \epsilon < 1$ 
there exists a positive vector $w_\epsilon$ such that the expected values for the weighted $\ell_1$-norm of the residuals satisfy
\[
\mathbb{E}(\|r^k\|_{1,w_\epsilon}) \leq \left(1- \alpha_\omega(\epsilon) \right)^{k} \| r^0  \|_{1,w_\epsilon},
\]
where $\alpha_\omega(\epsilon) =  \min_{j=1}^n p_j/\gamma_\omega(\epsilon) >0$, $\gamma_\omega(\epsilon) = (1-\rho_\omega+\epsilon)^{-1}$. The value of $\alpha_\omega(\epsilon)$ is maximized for the choice $p_i = {1}/{n}$, and 
the resulting value for $\alpha$ is $\alpha_\omega^{\opt}(\epsilon) = (1-\rho_\omega + \epsilon)/{n}$. 
\end{itemize}
\end{theorem}
\begin{proof}
We indeed have $\rho < 1$ since by Corollary~\ref{Jacobi_absolute_conv:cor} the operator norm  $\| | H|  \|_{1,u}$
is less than 1 for some vector $u > 0$. 
Assume first that $H$ is irreducible and let $u>0$ be the left Perron vector of $|H|$. Then we have for all $j=1,\ldots,n$
\[
 \sum_{i=1, i\neq j}^n |h_{ij}|{u}_i = \rho {u}_j,
 \]
 which for the weighted column sums of the matrix $H_\omega = (1-\omega)I + \omega H$ belonging to the relaxed iteration gives
 \[
 \frac{1}{u_j} \left( \sum_{i=1, i \neq j} \omega |h_{ij}|u_i +|1-\omega|u_j\right)  = \omega \rho + |1-\omega| = \rho_\omega.
 \]
 Since $\rho < 1$ we have that $\rho_\omega < 1$ for $\omega \in (0,\tfrac{2}{1+\rho})$. The result now follows applying Theorem~\ref{H_conv:thm} with the 
weight vector ${u}$, using the facts that all weighted column sums $\rho_j$ 
 are now equal to $\rho_\omega$ and that  $|\hat r^k_i| = \tfrac{1}{|a_{ii}|}
 | r^k_i|$, which gives the weights $w_i = u_i / |a_{ii}|$ in (i).
 
 If $H$ is not irreducible, the proof proceeds in exactly the same manner, chosing $u_\epsilon > 0$ as a vector for which $|H| u_\epsilon \leq (\rho+\epsilon) u_\epsilon$ ; see the discussion after Theorem~\ref{Perron_Frobenius:thm}. 
 \end{proof} 
 
For the same reasons as those explained after Theorem~\ref{randomized_H_conv_w:thm}, it is not possible to 
expand Theorem~\ref{relaxed_rGS_conv_th} to Gauss-Southwell unless we we know the Perron vector $u$ and include it into the greedy pick rule. We do not state this as a separate theorem. 
 
\subsection*{Numerical example}
We consider again the implicit Euler rule for the 
\linebreak
convection-diffusion equation \eqref{advec_diff:eq}, now with a non-vanishing and non-constant velocity field describing a re-circulating flow,
\[
(\nu(x,y), \mu(x,y))  = \sigma (\, 4x(x-1)(1-2y), -4y(y-1)(1-2x)\, ).
\]
We will consider the two choice $\sigma = 1$ (weak convection) and $\sigma = 400$ (strong convection); the diffusion coefficients $\alpha$ and $\beta$ are constant and equal to 1. With $N = 100$ and $\tau = 0.5h^2, h = \tfrac{1}{N+1}$, 
as in the example in section~\ref{hpd:sec}, the matrix $A=I+\tfrac{\tau}{2}B$ from \eqref{the_system:eq} is diagonally dominant for both $\sigma = 1$ and $\sigma = 400$. So we take the weight vector $u$ to have all components equal to 1 and we report results on the $\ell_1$-norm of the residuals for randomized Gauss-Seidel and Gauss-Southwell as an illustration of Theorem~\ref{GS_H_conv2:thm}. Since this time we are interested in the residuals rather than in the errors, it is not mandatory to fix the right hand side such that we know the solution, but to stay in line with the earlier experiments we actually did so with the solution being again the discretization of $xy(1-x)(1-y)$.

\af{The top row of} Figure~\ref{H:fig} displays, as before, these $\ell_1$-norms only after every $n$ relaxations, considered as one iteration. For both values of $\sigma$ we take the probabilities $p_i$ from \eqref{opt_prob_u:eq} in randomized Gauss-Seidel and the greedy pick rule \eqref{greedy_pick2:eq} for Gauss-Southwell, so that Theorem~\ref{GS_H_conv2:thm} applies, and the plots also report the bounds for the 1-norm of the residual given in that theorem.

\begin{figure}
\centerline{\includegraphics[width=0.5\textwidth]{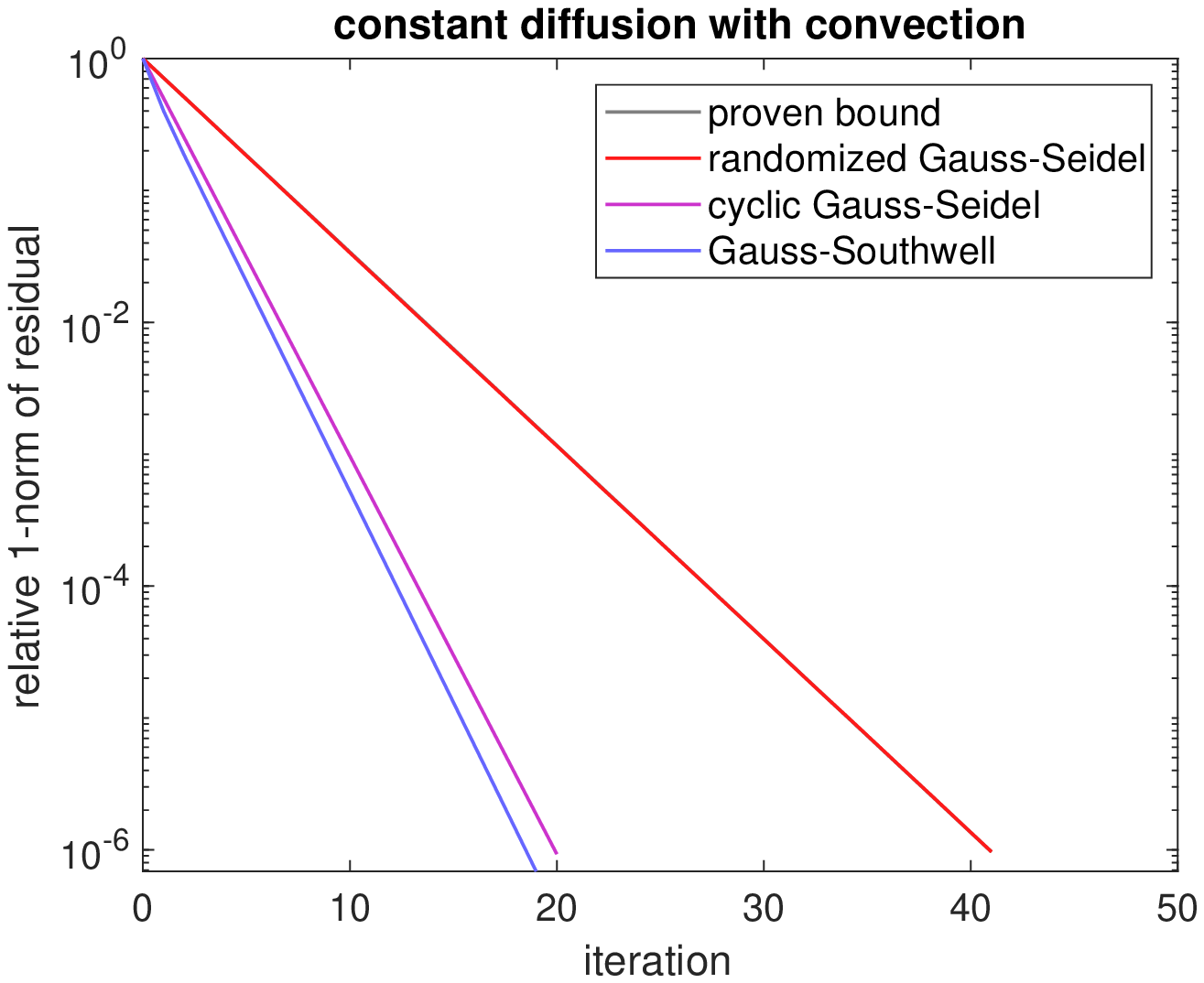}
\includegraphics[width=0.5\textwidth]{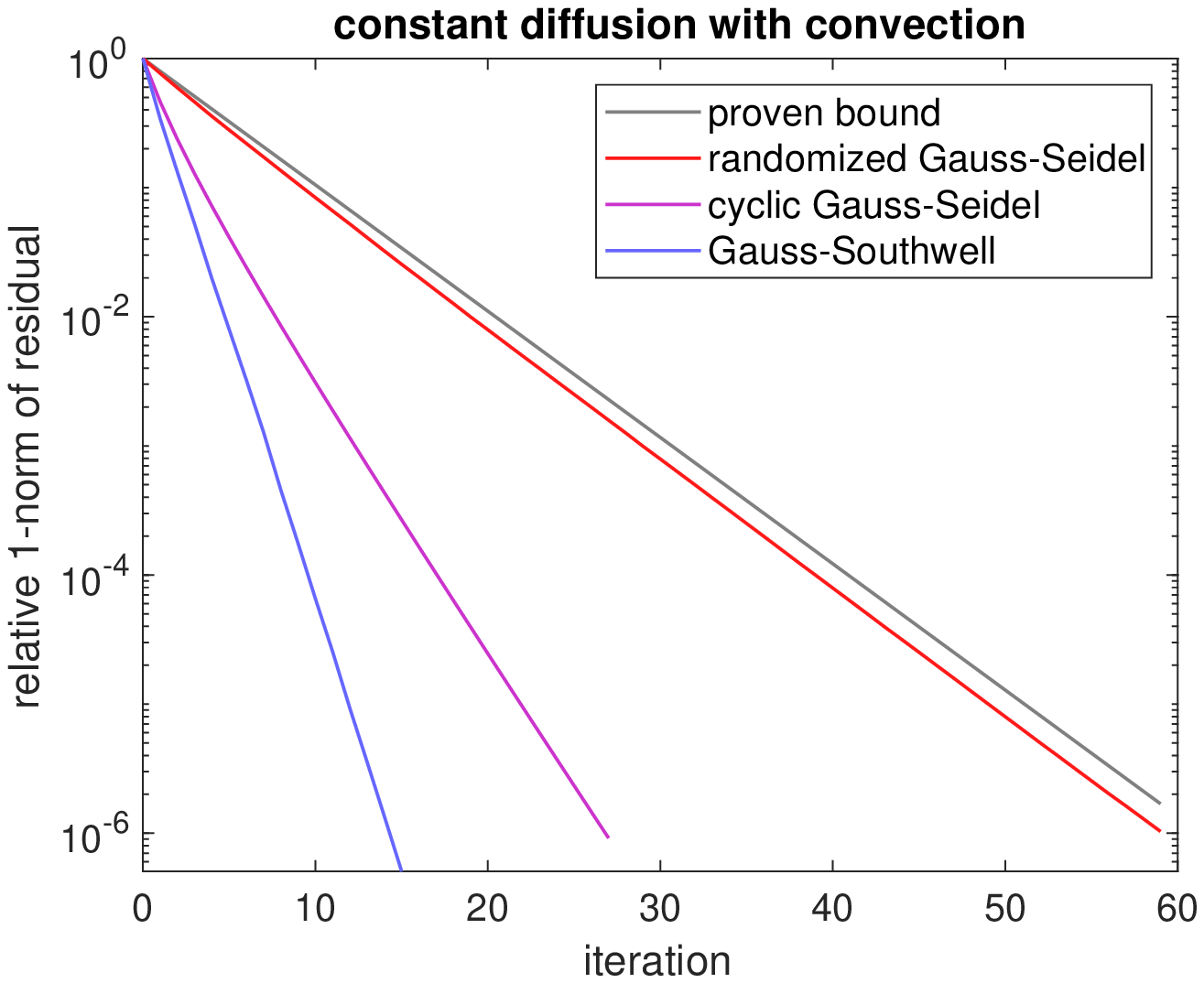}
}
\centerline{\includegraphics[width=0.5\textwidth]{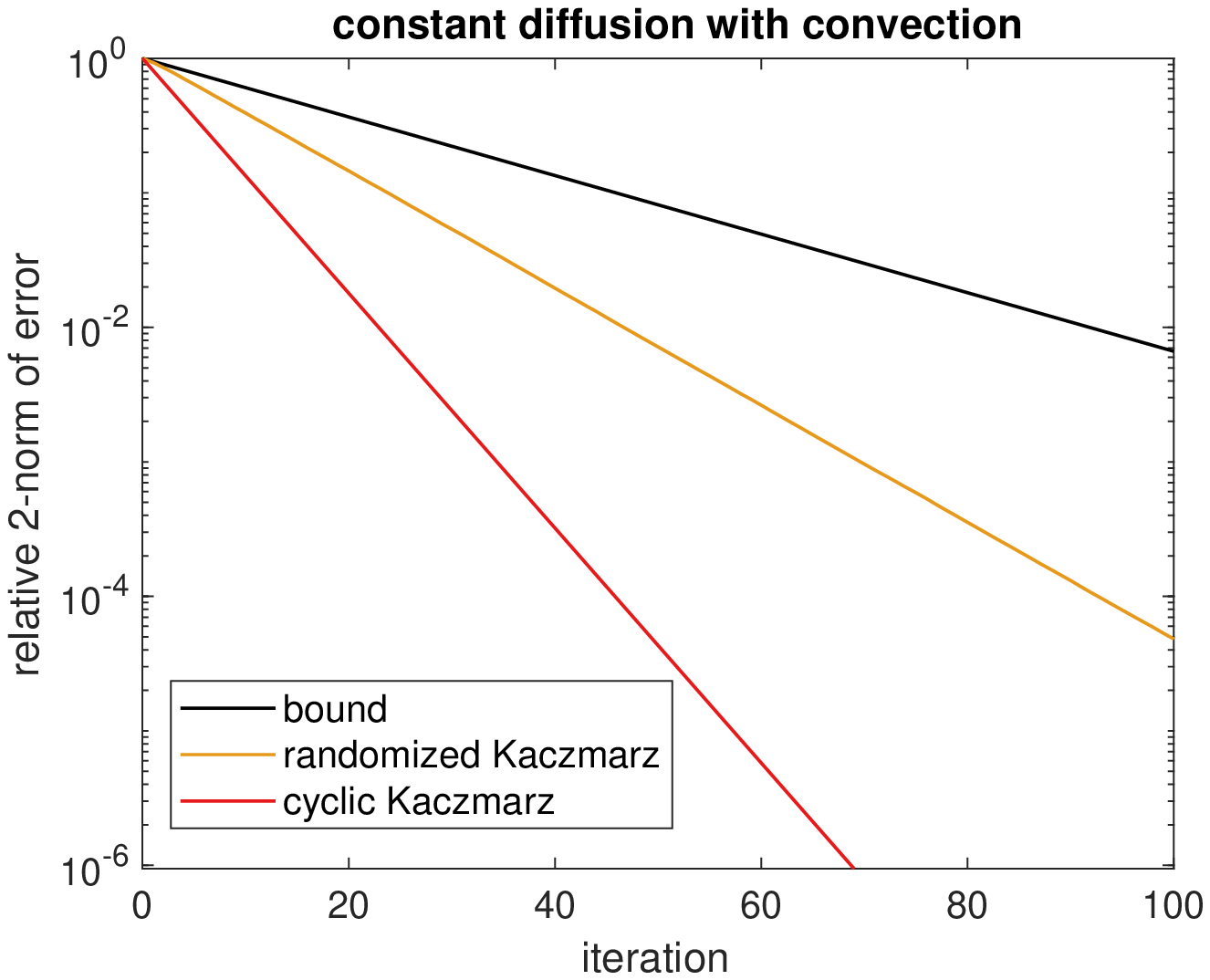}
\includegraphics[width=0.5\textwidth]{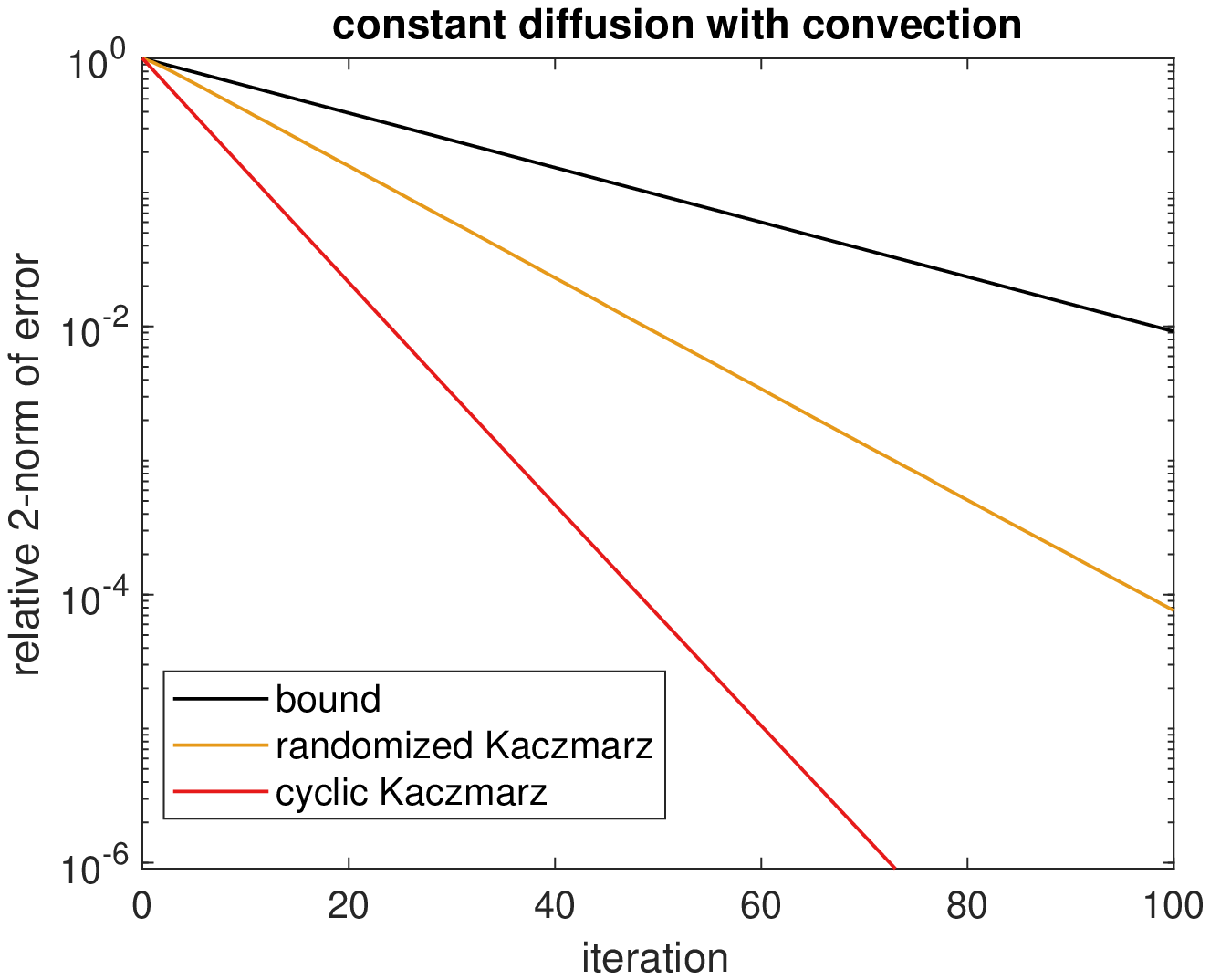}
}
\caption{\af{Comparison of Gauss-Seidel type (top row) and Kaczmarz type (bottom row) methods}. \textit{Left:} weak convection, $\sigma = 1$. \textit{Right:} strong convection, $\sigma = 400$.  Top-down order in the legend corresponds to top-down order of the {plotted lines}. \label{H:fig}}
\end{figure}

 The \af{plots in the top row} show a similar behavior of the different relaxation methods as in the hpd case, with Gauss-Southwell being fastest, especially for strong convection, and randomized Gauss-Seidel converging roughly half as fast as cyclic Gauss-Seidel---and this for both values of $\sigma$. As opposed to the hpd case, the theoretical bounds are now much closer to the actually observed convergence behavior of the randomized relaxations. For $\sigma = 1$ the bounds are actually that close that in the graph they are hidden behind the reported residual norms. For $\sigma=400$ the bounds can be distinguished from the residual norms in the plots, but they are still remarkably close.  
 
 \af{The bottom row of Figure~\ref{H:fig} contains results for the cyclic and the randomized Kaczmarz methods. 
 With $a_i^*$ denoting the $i$th row of $A$, and $a_i$ the  corresponding column vector, 
relaxing component $i$ in Kaczmarz amounts to the update
 \[
 x^{k+1} = x^{k} + \frac{b_i-a_i^*x^k}{\|a_i^*\|_2^2}a_i.
 \]
 We report these results since they allow for a comparison, 
 randomized Kaczmarz being one of the most prominent randomized 
 system solvers for non-symmetric systems. }

\af{The plots report the 2-norm of the error for which we know a bound given by the right hand side in the inequality
 \[
 \|x^k-A^{-1}b\|_2^2 \leq (1-\alpha^{\K})^n \|x^0-A^{-1}b\|_2^2 \enspace \mbox{with } \alpha^{\K} = \frac{\sigma_{\min}(A)^2}{\|A\|_F^2},
 \]
 where $\sigma_{\min}(A)$ is the smallest non-zero singular value of $A$. This bound holds if 'optimal' probabilities $p_i$ are chosen, 
see \cite{StrohmerVershyin2009}, as $
 p_i = \|a_i^*|_2^2/\|A\|_F^2$, 
 which is what we did for these experiments. 
 }
 
 \af{
 We see that as with Gauss-Seidel, the cyclic algorithms converges 
 approximately twice as fast than the randomized algorithm. The theoretical bounds are less sharp than for 
 the Gauss-Seidel methods. For both matrices, the cyclic and 
 randomized Kaczmarz methods converge significantly slower than 
 their Gauss-Seidel counterparts while, moreover, one relaxation 
 in Kaczmarz needs approximately twice as many operations than in 
 Gauss-Seidel. Since the plots on the top row report $\ell_1$-norms 
 rather than 2-norms of the error, let us just mention that in the 
 weak convection case, $\sigma = 1$, the relative 2-norm of the 
 residual in the randomized Gauss-Seidel run from the top row of 
 Figure~\ref{H:fig} is $1.22\cdot10^{-6}$ at iteration 41, and in 
 the strong convection case, $\sigma = 400$, it is $1.65\cdot10^{-6}$ at iteration 60.      
On the other hand, randomized Kaczmarz failed to converge to the desired tolerance of $10^{-6}$
within 100 iterations.
}

\section{Conclusion} 
We developed theoretical convergence bounds for both randomized and greedy pick relaxations for 
nonsingular linear systems. While we mainly reviewed results for the Gauss-Seidel 
relaxations in the case of a Hermitian positive definite matrix $A$, we presented several new convergence results for 
\af{nonsymmetric matrices in}
the case where the iteration matrix has a weighted $\ell_1$-norm less than 1. From this we could deduce several convergence results for randomized Gauss-Seidel and Gauss-Southwell relaxations for H-matrices. We also presented results which show how to choose the probability distributions (in case of randomized relaxations) or the greedy pick rule (in case of greedy iterations) which minimize 
our convergence bounds.   
Numerical experiments illustrate our theoretical results
\af{and also show that the methods analyzed are faster than Kaczmarz for square matrices}.
 
\paragraph*{Acknowledgement.} We want to thank Karsten Kahl from University of Wuppertal for sharing his Matlab implementations constructing convection-diffusion matrices with us.
We also thank Vahid Mahzoon {from Temple University} for some preliminary experiments which helped guide our thinking. 
\af{We express our gratitude to the two reviewers for their comments and questions, which helped improve our presentation.}

\bibliographystyle{plainnat}
\bibliography{random.v2}

\end{document}